\documentclass[a4paper,12pt,leqno]{amsart}

\usepackage{graphicx}
\usepackage{pinlabel} 
\usepackage{amsmath}
\usepackage{amsfonts}
\usepackage{amssymb}
\usepackage{mathtools}
\usepackage[all,cmtip]{xy}
\usepackage{amsthm}
\usepackage{tikz-cd}
\usepackage{comment}
\usepackage{enumerate}
\usepackage{url}
\usepackage{epsfig}
\usepackage[utf8]{inputenc}
\usepackage[capitalise]{cleveref}
\usepackage{geometry}
\makeatletter
\providecommand\@dotsep{5}
\def\listtodoname{List of Todos}
\def\listoftodos{\@starttoc{tdo}\listtodoname}
\makeatother

\newtheorem{theorem}{Theorem}[section]
\newtheorem{proposition}[theorem]{Proposition}
\newtheorem{corollary}[theorem]{Corollary}
\newtheorem{lemma}[theorem]{Lemma}

\newtheorem{step}{Step}

  \theoremstyle{definition}
\newtheorem{definition}[theorem]{Definition}

\newtheorem{remark}[theorem]{Remark}

\newcommand{\Z}{\mathbb{Z}}

\newcommand{\calF}{{\mathcal F}}

\newcommand{\calG}{{\mathcal G}}
\newcommand{\calH}{{\mathcal H}}

\newcommand{\nbeq}{\begin{equation}}
\newcommand{\neeq}{\end{equation}}
\newcommand{\beq}{\begin{equation*}}
\newcommand{\eeq}{\end{equation*}}

\DeclareMathOperator{\vcd}{vcd}

\DeclareMathOperator{\gd}{gd}

\DeclareMathOperator{\mcg}{Mod}

\DeclareMathOperator{\Aut}{Aut}
\DeclareMathOperator{\Diff}{Diff}

\usepackage{verbatim}

\usepackage{thmtools}
\usepackage{thm-restate}

\begin{document}

\title[Commensurators of mapping class groups]
{Commensurators  of  abelian subgroups and the virtually abelian dimension of mapping class groups}




\dedicatory{To Daniel Juan-Pineda on the occasion of his 60th birthday.}

\author[Rita Jiménez Rolland]{Rita Jiménez Rolland}

\address{Instituto de Matemáticas, Universidad Nacional Autónoma de México. Oaxaca de Juárez, Oaxaca, México 68000}
\email{rita@im.unam.mx}

\author[Porfirio L. León Álvarez ]{Porfirio L. León Álvarez}

\address{Instituto de Matemáticas, Universidad Nacional Autónoma de México. Oaxaca de Juárez, Oaxaca, México 68000}
\email{porfirio.leon@im.unam.mx}

\author[Luis Jorge Sánchez Saldaña]{Luis Jorge S\'anchez Salda\~na}
\address{Departamento de Matemáticas, Facultad de Ciencias, Universidad Nacional Autónoma de México, Mexico}
\email{luisjorge@ciencias.unam.mx}


\date{}



\begin{abstract} 


Let $\mcg(S)$  be the  mapping class group of a compact connected orientable surface $S$, possibly with punctures and boundary components,   with negative Euler characteristic. We prove that for any  infinite virtually abelian subgroup $H$ of $\mcg(S)$, there is a subgroup $H'$ commensurable with $H$ such that the commensurator of $H$ equals the normalizer of $H'$. As a consequence we give, for each $n \geq 2$, an upper bound for the geometric dimension of $\mcg(S)$ for the family of abelian subgroups of rank bounded by $n$. These results generalize work by Juan-Pineda--Trujillo-Negrete and Nucinkis--Petrosyan for the virtually cyclic case.
\end{abstract}
\maketitle

\tableofcontents
\section{Introduction}

\noindent Given any group $G$ and two subgroups $H$ and $K$ of $G$, we say that $H$ and $K$ are {\it commensurable} if $H\cap K$ has finite index in both $H$ and $K$.  The {\it commensurator of $H$ in $G$}, which we denote by $N_G[H]$ throughout the paper, is by definition the set of all elements $g\in G$ such  that $H$ and $gHg^{-1}$ are commensurable. It is clear from the definition that the {\it normalizer} $N_G(H)$ of $H$ in $G$ is contained in $N_G[H]$ and, as shown in \cite[Example 2.6]{LW12}, this inclusion might be strict. The commensurator $N_G[H]$ only depends on the commensuration class of $H$, that is, if $H'$ is group commensurable with $H$, then $N_G[H]=N_G[H']$. In the literature we can find several results of the following type: 

\begin{center}
 {\it Given a group $G$ and a subgroup $H$ of $G$, }
\item[$(\star)$] {\it
the commensurator of $H$ in $G$ can be realized as the normalizer in $G$ of a subgroup $H'$ in the same commensuration class of $H$. }

\end{center}

Among the examples of groups and subgroups exhibiting this phenomenon are the following: any subgroup of a virtually polycyclic group \cite[Theorem~10]{MR4138929}, infinite virtually cyclic subgroups of the mapping class group of an orientable compact surface with negative Euler characteristic \cite[Proposition 4.8]{JPT16} and  \cite[Proposition 5.9]{Nucinkis:Petrosyan}, certain abelian subgroups of $\mathrm{Out}(F_n)$ \cite{Guerch2022}, infinite virtually cyclic subgroups of linear groups \cite[Lemma 3.2]{DP15}, parabolic subgroups and infinite virtually cyclic subgroups not contained in a parabolic subgroup of relatively hyperbolic groups \cite[Theorem 2.6]{LO07}.

Let $S$ be a connected compact surface with possibly finitely many punctures. The {\it mapping class group of $S$} is the group $\mcg(S)$ of isotopy classes of (orientation preserving, if $S$ is orientable) diffeomorphisms of $S$ that restrict to the identity on the boundary $\partial S$. In this paper we show that 
all infinite virtually abelian subgroups of $\mcg(S)$ satisfy the aforementioned phenomenon ($\star$), when the surface $S$ is orientable and has negative Euler characteristic. More precisely, bringing together \Cref{norrmalizer:commensurator:equal} and  \Cref{commensu:equal:normalizer} 
we obtain the following theorem. 



\begin{theorem}\label{thm:normalizer:commensurator:general:case}
 Let $S$ be a connected compact orientable surface possibly with a finite number of punctures and with negative Euler characteristic. Let $H$ be a virtually abelian  subgroup of $\mcg(S)$ of rank $k\ge 1$. Then there is a subgroup  
 $H'$ of $\mcg(S)$ which is commensurable with $H$ 
 such that \[N_{\mcg(S)}[H]=N_{\mcg(S)}[H']=N_{\mcg(S)}(H').\]
\end{theorem}
 
 \begin{remark}\label{remark1:intro} If $S$ is a closed surface possible with finitely many punctures, the subgroup $H'$ can be taken to be of finite index in $H$ in the statement of \cref{thm:normalizer:commensurator:general:case} above; see \Cref{norrmalizer:commensurator:equal}. Furthermore, for any subgroup $\Gamma$ of $\mcg(S)$, from \cref{normarlizer:equal:Commen:subgroup}, it follows that the commensurator of any infinite virtually abelian subgroup $H$ of $\Gamma$  can be realized as the normalizer of a finite index subgroup $H'$ of $H$. 
 Some examples of such subgroups $K$ that may be of interest are:
 \begin{itemize}
     \item  The mapping class groups $\mcg(N)$, where $N$ is a non-orientable closed surface, possibly with punctures, with negative Euler characteristic. This group can be realized as a subgroup of the mapping class group $\mcg(M)$ of the orientable double cover $M$ of $N$; see \Cref{non-orientable:surfaces} and \Cref{Prop:normalizers:equal:comm:nonorien}.

     \item Any right-angled Artin group (RAAG) $\Gamma$, or more generally, any $\Gamma$ that has a subgroup of finite index that embeds in a RAAG. From \cite[Corollary 5.2]{MR3151642} any such $\Gamma$ embeds in $\mcg(S)$  for some closed  orientable surface $S$ of  sufficiently large genus; see also \cite[Corollary 3]{MR2327038}.  

     \item The braid group $B_n(S)$ of a connected closed surface $S$ with negative Euler characteristic. There exists a monomorphism $\psi\colon B_n(S)\to \mcg(S,n)$, where $\mcg(S,n)$ denotes the mapping class group of the closed surface $S$ with $n$ punctures; see for example,  \cite[Section 2.4]{MR3382024}.
     
 \end{itemize}

 \end{remark}

\begin{remark}  From \Cref{lemma:com:norm:rank1}, \Cref{norrmalizer:commensurator:equal} and  \cref{auxilar:proposition:boundary}, it follows that  any central extension $\Gamma$ of  a subgroup  of $\mcg(S)$ by a finitely generated free abelian group  satisfies that condition $(\star)$ holds for every finitely generated free abelian subgroup of $\Gamma$. 
This allows us to promote \Cref{norrmalizer:commensurator:equal} from closed surfaces to \Cref{commensu:equal:normalizer} for surfaces with non-empty boundary using the capping homomorphism (see \cref{NON-EMPTY}).  

As a consequence of the previous paragraph and \Cref{remark1:intro},  for fixed $n\geq3$, Artin groups of finite type $A_n$ and $B_n=C_n$ are a central extension of a finite index subgroup of the mapping class group $\mcg(S^2,n+2)$ of the $(n+2)$-punctured sphere $S^2$ by an infinite cyclic group  and the Artin groups of affine type $\tilde{C}_{n-1}$ and  $\tilde{A}_{n-1}$ can be realized as a finite index subgroup of $\mcg(S^2,n+2)$ (see for example \cite[Section2]{MR2150887} and references therein). Hence, free abelian subgroups of such groups satisfy $(\star)$.

\end{remark}

Concerning the structure of abelian subgroups of $\mcg(S)$, it was shown by Birman--Lubotzky--McCarthy in \cite[Theorem A]{BirmanLubotzkyMcCarthy83} that they are finitely generated and the maximum possible rank of an abelian subgroup was computed for closed orientable surfaces $S$ with finitely many punctures and negative Euler characteristic (see also \cref{maximal:rank:abelian:subgroup}).  Analogue results for the non-orientable case were proved recently by Kuno in \cite{kuno2019}.

One of the fundamental ingredients in the proof of \cref{thm:normalizer:commensurator:general:case} is the reduction system theory developed by Ivanov in \cite[Chapter 7]{Ivanov:subgroups} for reducible subgroups of $\mcg(S)$. The proof of the rank $k=1$ case,  when the surface $S$ is orientable,  uses the canonical reduction systems of reducible elements of $\mcg(S)$. It is due to the work of Juan-Pineda--Trujillo-Negrete in \cite{JPT16}, and Nucinkis--Petrosyan in \cite{Nucinkis:Petrosyan}; see \Cref{section:normalizers} for further references. Although, we follow a very similar strategy and we adopt their notation, we had to overcome some difficulties that arise when we consider a  virtually abelian subgroup of rank  $k\geq 2$.

The main technical step towards the proof \Cref{thm:normalizer:commensurator:general:case} is to describe the structure of both $N_{\mcg(S)}(H)$ and $N_{\mcg(S)}[H]$ as  a central extension of a free abelian subgroup by a group 
that has finite index in a group of the form $\mcg(\hat S)\times A$, where $\hat S$ is a non-connected compact surface and $A$ is a virtually abelian group. This is done in \Cref{Normalizer:commensurator:free:abelian:subgroup}, and it is a generalization of \cite[Proposition 4.12]{JPT16} and \cite[Proposition 5.6]{Nucinkis:Petrosyan}. It is worth mentioning that the proof of \cite[Proposition 4.12]{JPT16} uses the fact that the centralizer of an infinite cyclic subgroup has finite index in the normalizer of the same group, which is an easy observation. On the other hand, when $H$ is a free abelian subgroup of rank at least 2 of a group $G$, it is not necessarily true that the centralizer $C_G(H)$ has finite index in $N_G(H)$. This is one of the technical difficulties that we sorted out in the proof of \Cref{Normalizer:commensurator:free:abelian:subgroup}. In fact, as a consequence of \Cref{Normalizer:commensurator:free:abelian:subgroup} we obtained that, when $G=\mcg(S)$ for $S$ a closed orientable surface, the centralizer of certain free abelian subgroups have finite index in the corresponding normalizer (see \Cref{cor:centra:norma}).

As an application of \Cref{thm:normalizer:commensurator:general:case} and the short exact sequences given in \Cref{Normalizer:commensurator:free:abelian:subgroup}, we compute an upper bound for the geometric dimension of $\mcg(S)$ for the family of abelian subgroups of rank bounded by $n$, for each $n\geq 2$. To state this result in full detail we introduce first some definitions and notation. 


Fix a group $G$ and a {\it family} $\calF$ of subgroups of $G$, i.e. a  non-empty collection $\calF$ of subgroups of $G$ that is closed under conjugation and under taking subgroups.  We say that a $G$-CW-complex $X$ is a \emph{model for the classifying space} $E_{\calF} G$ if every isotropy  group of $X$ belongs to the family $\calF$ and the fixed point set $X^H$ is contractible whenever  $H$ belongs to $\calF$. It can be shown that a model for the classifying space   $E_{\calF} G$ always exists and it is unique up to $G$-homotopy equivalence. The \emph{$\calF$-geometric dimension} of $G$   is by definition
$$\gd_{\calF}(G)=\min\{k\in \mathbb{N}| \text{ there is a model for } E_{\calF}G \text{ of dimension } k \}.$$
Let $n\geq0$ be an integer. A group is said to be \emph{virtually $\mathbb{Z}^n$} if it contains a subgroup of finite index isomorphic to $\mathbb{Z}^n$. Define the family \[\calF_n=\{H\leq G| H \text{ is virtually } \mathbb{Z}^r \text{ for some } 0\le r \le n \}.\] The families $\calF_0$  and $\calF_1$ are the families of finite and virtually cyclic subgroups, respectively, and are relevant due to their connection with the Farrell-Jones and Baum-Connes isomorphism conjectures, see for instance \cite{LR05}. In \cite{BB19} Bartels--Bestvina  proved that the mapping class group of an orientable surface of finite type satisfies the Farrell-Jones conjecture. The families $\calF_n$  have been recently studied by several people, for instance there are computations for $\mathrm{CAT}(0)$-groups \cite{tomasz,HP20} and for free abelian groups \cite{Nucinkis:Moreno:Pasini,MR4472883}. Also, the second named author has an article in progress where he computes the virtually abelian dimension of braid groups,  RAAG's, and virtually finitely generated free abelian groups. In this paper we contribute to the existing literature by obtaining an upper bound for the $\calF_n$-geometric dimension of $\mcg(S)$ for each $n\geq 2$.

\begin{restatable}{theorem}{dimensionbound}
\label{Thm:dimension:bound:main}
Let $S$ be a connected compact orientable surface possibly with a finite number of punctures. Assume that $S$ has negative Euler characteristic.  Then for all $ n\in \mathbb{N}$ we have  $$\gd_{\calF_n}(\mcg(S))\leq \vcd (\mcg(S)) +n.$$
\end{restatable}



The cases $n=0,1$ are the induction basis for the proof of \Cref{Thm:dimension:bound:main}.  The equality is known to hold in the case $n=0$ by work of Harer \cite{Ha86}  and Aramayona--Martínez-Pérez  \cite{AMP14}. The upper bound for the case $n=1$ was  obtained by  Nucinkis--Petrosyan \cite{Nucinkis:Petrosyan}, where the equality was proved for closed orientable surfaces. The virtual cohomological dimension of the mapping class group of an orientable surface was computed by Harer in \cite{Ha86}. 

For $S$ a closed orientable surface of genus $g$ and with $b\geq 0$ punctures, by \cite[Theorem A]{BirmanLubotzkyMcCarthy83}  every free abelian subgroup of $\mcg(S)$ has rank at most $r=3g-3+b$. Then for all $n\ge r$ we have that $\calF_n=\calF_r$ and $\gd_{\calF_n}(\mcg(S))=\gd_{\calF_r}(\mcg(S))$. It follows from \Cref{Thm:dimension:bound:main} that $$\gd_{\calF_n}(\mcg(S))\leq \vcd(\mcg(S))+r\text{\ \  for all \ \ }n\geq r.$$

On the other hand, consider a connected closed non-orientable surface $N$, possibly with punctures,  with negative Euler characteristic. The group $\mcg(N)$ can be realized as a subgroup of the mapping class group of $M$, the orientable double cover of $N$; see \cref{non-orientable:surfaces}. Since geometric dimensions are monotone it follows from \cref{Thm:dimension:bound:main} that $$\gd_{\calF_n\cap\mcg(N)}(\mcg(N))\leq \vcd(\mcg(M))+n\text{\  \  for all\  \  }n\in \mathbb{N}.$$

\addtocontents{toc}{\protect\setcounter{tocdepth}{0}}
\subsection*{Outline of the paper.} In \Cref{MCG} preliminaries and notation about mapping class groups are recalled. In particular, the definition of the cutting homomorphism associated to a collection of disjoint simple curves and the main properties of canonical reduction systems for reducible subgroups of $\mcg(S)$ following \cite{Ivanov:subgroups}. \Cref{sec:pushouts} is devoted to introduce a particular case of the Lück--Weiermann construction to promote a classifying space of a group $G$ with respect to the family $\mathcal{F}_{n}$ to a classifying space for the family $\mathcal{F}_{n+1}$; this constructions allow us to perform an induction argument in the proof of \Cref{Thm:dimension:bound:main}. \Cref{section:normalizers} is the longest and most technical section, and it deals with the structure of commensurators, normalizers and centralizers of abelian subgroups of $\mcg(S)$ which are used to obtain \Cref{thm:normalizer:commensurator:general:case}. Finally, in \Cref{section:virtually:abelian:dimension} we prove \Cref{Thm:dimension:bound:main} using the results from \Cref{section:normalizers}.

\subsection*{Acknowledgements.}
The second author was supported by a doctoral scholarship of the Mexican Council of Science and Technology (CONACyT). We are grateful for the financial support of DGAPA-UNAM grant PAPIIT IA106923
and CONACyT grant CF 2019-217392.

\addtocontents{toc}{\protect\setcounter{tocdepth}{2}}

\section{Preliminaries on mapping class groups}\label{MCG}

In this section we recall some background on mapping class groups that will be needed to obtain our results. We use the notation from \cite{JPT16} and \cite{Nucinkis:Petrosyan}; for further details we refer the reader to \cite{Ivanov:subgroups} and \cite{FM12}.

Let $S$ be a connected compact oriented surface with finitely many punctures. The {\it mapping class group of $S$} is the group $\mcg(S)$ of isotopy classes of orientation preserving diffeomorphisms of $S$ that restrict to the identity on the boundary $\partial S$. If $\Diff_0(S,\partial S)$ denotes the subgroup of $\Diff^+(S,\partial S)$ of orientation preserving diffeomorphisms that are isotopic to the identity then  
$$\mcg(S)=\Diff^+(S,\partial S)/\Diff_0(S,\partial S).$$


For $m\geq 2$, the {\it level $m$ congruence subgroup $\mcg(S)[m]$ of} $\mcg(S)$ is defined as the kernel of the natural homomorphism $\mcg(S)\rightarrow \Aut\left(H_1(S,\mathbb{Z}/m\mathbb{Z})\right)$ given by the action of diffeomorphisms in the homology of the surface. It is a finite index subgroup of $\mcg(S)$ and for $m\geq 3$ it is torsion free (see Theorem 3 of \cite{Ivanov:subgroups}). 

A simple closed curve $\alpha$ in $S$ is essential if it is not homotopic to a point, a puncture  or  a boundary component of $S$. We denote the isotopy class of such $\alpha$ by $[\alpha]$ and by $V(S)$ the set of isotopy classes of essential curves in $S$.  The {\it complex of curves}  $C(S)$ is the simplicial complex with set of vertices $V(S)$  and with a $k$-simplex given by a set of $k+1$ vertices in $V(S)$ with mutually disjoint essential curves representatives.
The group $\mcg(S)$ acts on $V(S)$ by $g[\alpha]=[\varphi(\alpha)]$, where $\varphi$ a diffeomorphism of $S$  that represents  $g\in\mcg(S)$. This action takes simplexes into simplexes, hence $\mcg(S)$ also acts on $C(S)$. 


\subsection{Reduction systems and the cutting homomorphism}
For what follows in Section \ref{MCG}  we assume that the surface $S$ has empty boundary and negative Euler characteristic. 

\begin{definition} A subgroup $H$ of $\mcg(S)$ is called {\it reducible} if there is a nonempty simplex $\sigma$ of $C(S)$ such that $h\sigma=\sigma$ for every $h\in H$. Such simplex $\sigma$  is called a {\it reduction system for} $H$.  If no nonempty reduction system exists,  $H$ is said to be {\it irreducible}.  
\end{definition}

The Nielsen-Thurston classification theorem classifies the elements of $\mcg(S)$ in reducible, periodic or pseudo-Anosov (see for example \cite[Theorem 13.2]{FM12}). An element $f\in \mcg(S)$ is reducible if the group $\langle f\rangle$ is reducible; otherwise $f$ is irreducible. Among the irreducible elements, those of finite order are periodic and those of infinite order are pseudo-Anosov.

Let $\sigma=\{[\alpha_1],\ldots,[\alpha_n]\}$ be a reduction system for a group $H$ with $\{\alpha_1,\ldots, \alpha_n\}$ a set of mutually disjoint essential curves representatives.  Consider the surface that results of cutting $S$ along this reduction system  $\hat{S_{\sigma}}=S\setminus\cup_{k=1}^n \alpha_k=\sqcup_{j=1}^l\hat{S_j}$,  where each $\hat{S_j}$ is a connected subsurface of $\hat{S_{\sigma}}$. Let us denote by $\Omega_j$ the set of punctures in $\hat{S_j}$ coming from the simple closed curves $\alpha_k$ that were `cut' from $S$. The inclusions $\hat{S_j}\hookrightarrow \hat{S}_{\sigma}$ induce a monomorphism $\prod_{j=1}^l\mcg(\hat{S_j},\Omega_j) \hookrightarrow \mcg(\hat{S}_{\sigma})$, where  $\mcg(\hat{S_j},\Omega_j)$ is the pure mapping class group of $\hat{S_j}$ that fixes pointwise the set of punctures $\Omega_j$.  


Define $\mcg(S)_{\sigma}=\{g\in\mcg(S)|g(\sigma)=\sigma\}$ and $\mcg(S)_{\sigma}^0$ its finite index subgroup consisting of elements that fix each curve $\alpha_k$ with orientation. There is a well-defined homomorphism, often called the {\it cutting homomorphism}
\begin{equation}\label{eq:cutting:homom}
    \rho_{\sigma}\colon \mcg(S)_{\sigma}\to \mcg(\hat{S}_{\sigma})
\end{equation}
with $\ker \rho_{\sigma}=\langle T_{\alpha_1},\ldots,T_{\alpha_n}\rangle$, the free abelian group of rank $n$ generated by the Dehn twists along the curves $\alpha_1,\ldots,\alpha_n$. When restricted to the subgroup $\mcg(S)_{\sigma}^0$ the image of $\rho_{\sigma}$ surjects onto $\prod_{j=1}^l\mcg(\hat{S_j},\Omega_j)$. Let $\varphi_i:\prod_{j=1}^l\mcg(\hat{S_j},\Omega_j) \rightarrow \mcg(\hat{S_i},\Omega_i)$ denote the projection onto the $i$th-factor.  

\begin{remark}
If  $m\geq 3$ and $H$ is a subgroup of $\mcg(S)[m]$, then  $H\subset \mcg(S)_{\sigma}^0$ by \cite[Theorem 1.2]{Ivanov:subgroups} and  the group $\varphi_i\circ\rho_{\sigma}(H)$ is a subgroup of $\mcg(\hat{S_i},\Omega_i)$ (see also \cite[Section 7.5]{Ivanov:subgroups}).
\end{remark}

\begin{remark}\label{subsection:pre-images}[Pre-images of the cutting homomorphism]
Let $f\in\prod_{j=1}^l\mcg(\hat{S_j},\Omega_j) \subset\mcg(\hat{S}_{\sigma})$. Take a representative diffeomorphism $F\colon\hat{S_{\sigma}}\to\hat{S_{\sigma}}$  of $f$ such that, for every $1\leq j\leq l$, it restricts to a diffeomorphism $F|_{\hat{S}_j}\colon \hat{S}_j\to \hat{S}_j$ that is the identity in (disjoint) tubular neighborhoods $U_x$ of  each  puncture $x$ in $\Omega_j$. Let $S_j=\hat{S}_j-\sqcup_{x\in \Omega_j} U_x$ and consider it as a subsurface of $S$. Define the diffeomorphism $\tilde f_j:S\to S$ by extending $F|_{S_j}$  with the identity in $S\setminus S_j$. Then $\tilde f_j$ fixes each curve $\alpha_k$ in the simplex $\sigma$, since the support of $\tilde f_j$ is contained in $S_j$; see figure \ref{fig:Preimage}. Let $\tilde{f}:S\rightarrow S$ be the diffeomorphism given  by the composition $\tilde{f}_1\tilde{f}_2\cdots\cdots\tilde{f}_l$ of diffeomorphisms of $S$ with disjoint support. Therefore $\tilde f$ represents an element of $\mcg(S)_\sigma^0$  with image $f$ under the  cutting homomorphism $\rho_\sigma$.
\begin{figure}[h]
    \centering
 \includegraphics[width=6.4in]{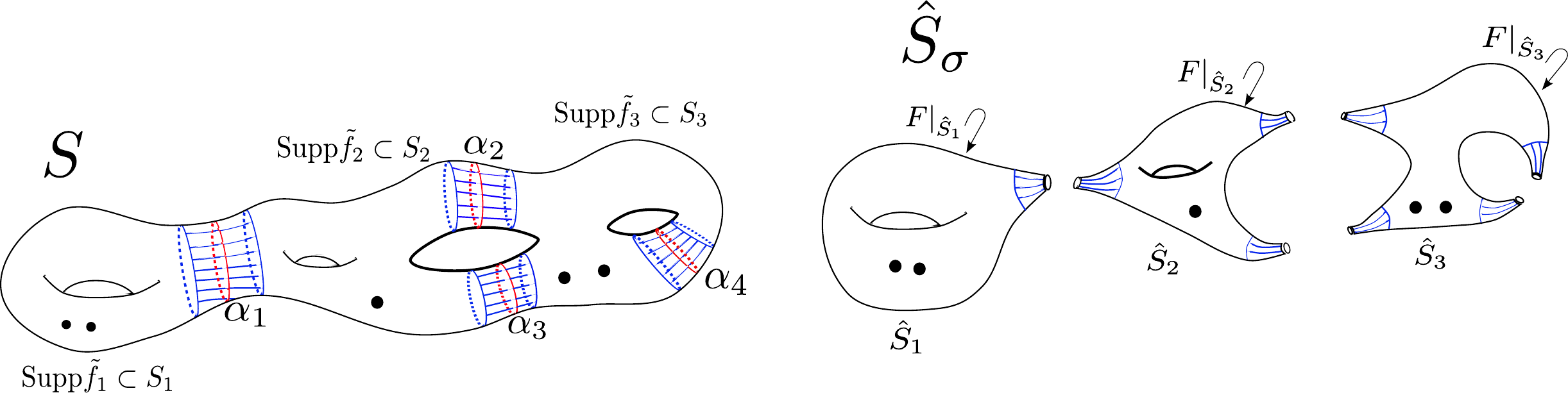} 
    \caption{\small Given the simplex $\sigma=\{\alpha_1,\alpha_2,\ldots \alpha_n\}$ of $C(S)$,  any mapping class $f$ in $\prod_{j=1}^l\mcg(\hat{S_j},\Omega_j)$ has a pre-image of the cutting $\rho_{\sigma}$ in $\mcg(S)_{\sigma}^0$, represented by a diffeomorphims $\tilde{f}:S\to S$ given by the composition $\tilde{f}_1\tilde{f}_2\ldots\tilde{f}_l$ of diffeomorphisms of $S$ with disjoint support.}
    \label{fig:Preimage}
\end{figure}

\end{remark}

\subsection{Canonical reduction systems and free abelian subgroups} Let $m\geq 3$ and $H$ be a subgroup of $\mcg(S)[m]$.  We denote by $\sigma(H)$ the {\it canonical reduction system for $H$} as defined by Ivanov in \cite{Ivanov:subgroups}.  We recall here the properties that will be needed for our arguments in the next sections, 
the precise definitions are stated in \cite[Section 7.2 and 7.4]{Ivanov:subgroups}.






\begin{lemma}\label{PropertiesRedSyst} \cite[Chapter 7]{Ivanov:subgroups}
 Let $m\geq 3$ and consider $H$  a subgroup of $\mcg(S)[m]$. The set $\sigma(H)$ is a reduction system (possibly empty) for $H$ and satisfies the following properties:
\begin{itemize}
\item[(i)] If $H'$ is a finite index normal subgroup of $H$, then $\sigma(H')=\sigma(H)$.
\item[(ii)] For $g\in\mcg(S)$, we have that $g\sigma(H)=\sigma(gHg^{-1})$.
\item[(iii)]  If  $H$ is nontrivial and reducible, then  $\sigma(H)\neq \phi$.
\end{itemize}
\end{lemma}

\begin{remark}
For a reducible element $f\in\mcg(S)[m]$, the simplex $\sigma(f)=\sigma(\langle f\rangle)$ is called a canonical reduction system for $f$ and it precisely the set of essential reduction classes in the sense of \cite{BirmanLubotzkyMcCarthy83}  (see also \cite[Section 13.2]{FM12}).


\end{remark}






Let $m\geq 3$ and let $H$ be a free abelian subgroup of $\mcg(S)[m]$ of rank at least $2$. By \cite[Corollary 8.6]{Ivanov:subgroups} an irreducible abelian subgroup of $\mcg(S)$ is virtually cyclic,  therefore $H$ must be reducible and $\sigma(H)\neq\phi$ by Lemma \ref{PropertiesRedSyst}(iii). 

The following result follows from \cite[Lemma 8.7, Corollary 8.5]{Ivanov:subgroups}. It gives us a tool analogous to the {\it canonical form} for a reducible  mapping class (see for example \cite[Corollary 13.3]{FM12}) that will be an important ingredient for our discussion below.



\begin{theorem}\label{Cor:action:components} Let $m\ge 3$ and let $H$ be  a free abelian  subgroup of $\mcg(S)[m]$ of rank  at least $2$. Then $H$ is reducible and has a nonempty canonical reduction system $\sigma=\sigma(H)$. Moreover, we can write $\hat{S}_{\sigma}=\bigsqcup_{i=1}^{a}\hat{S_i} \bigsqcup_{j=a+1}^{l}\hat{S_j}$ so that
\begin{itemize}
\item[a)] the group $\varphi_i\circ\rho_{\sigma}(H)$ is trivial  for all $1\le i\le a$, and
\item[b)] the group $\varphi_j\circ\rho_{\sigma}(H)$ is  infinite cyclic generated by a pseudo-Anosov mapping class for all $a+1\le j\le l$.
\end{itemize}
\end{theorem}

\begin{remark} To simplify the exposition, in our proofs below we will apply the results compiled in this section for the case $m=3$ (we could in principle have used any fixed $m\geq 3$).  
\end{remark}

The rank of free abelian subgroups of $\mcg(S)$ is bounded above (\cite[Theorem A]{BirmanLubotzkyMcCarthy83}, \cite[Lemma 8.8]{Ivanov:subgroups}).

\begin{theorem}[Maximal rank for free abelian subgroups]\label{maximal:rank:abelian:subgroup} Let $H$ be a  abelian subgroup of $\mcg(S)$. Then $H$ is finitely generated  and the rank of H is at most $3g-3+b$, where $g$ is the genus of the orientable surface $S$ and $b$ is the number punctures of $S$.
    
\end{theorem}

\section{Push-out constructions for classifying spaces}\label{sec:pushouts}

 In this subsection we describe a particular case, which is convenient for the purposes of the present paper, of  the Lück and Weiermann push-out construction.

 Fix $n\geq 1$ and $G$ a group. Recall that two subgroups $H$ and $K$ of $G$ are {\it commensurable} if $K\cap H$ has finite index in both $K$ and $H$. We define an equivalence $\sim$ relation on $\calF_n-\calF_{n-1}$ as follows:
 \[H\sim K \iff H\text{ is commensurable with }K.\]

 It is not hard to see that this equivalence relation satisfies the following two properties:
 
\begin{enumerate}[a)]
    \item If  $H, K \in \calF_n-\calF_{n-1}$ with $H\subseteq K$, then $H\sim K$; this is true because $H$ and $K$ are virtually abelian of the same rank.
    \item If $H, K \in \calF_n-\calF_{n-1}$ and $g\in G$, then $H\sim K$ if and only if $gHg^{-1} \sim gKg^{-1}$; this follows directly from the definition of commensurability.
\end{enumerate}

 For a subgroup $H$ of $G$ we set the following notation:
 \begin{itemize}
     \item Denote by $[H]$ the commensuration class of $H$ in $G$.
     \item $N_G[H]=\{g \in G | gHg^{-1}\sim H\}$, this is the so-called \emph{commensurator} of $H$ in  $G$.
     \item $\calF_{n}[H]=\{K\subseteq N_G[H]| K\in  \calF_n-\calF_{n-1}, [K]=[H]\}\cup (\calF_{n-1}\cap N_G[H])$, this set is a family of subgroups of $N_G[H]$. 
     \item Denote by $(\calF_n-\calF_{n-1})/\sim$ the set of equivalence classes with respect to $\sim$.
 \end{itemize}

 \begin{theorem}\cite[Theorem 2.3]{LW12}\label{Luck:weiermann}  Let $G$ be a group. Let $I$ be a complete set of representatives of conjugation classes in $(\calF_n-\calF_{n-1})/\sim$. Choose arbitrary $N_G[H]$-CW-models for $E_{\calF_{n-1}\cap N_{G}[H]}N_{G}[H]$ and $E_{ \calF_{n}[H]}N_{G}[H]$, and an arbitrary model for  $E_{\mathcal{F}_{n-1}}G$. Consider the  following $G$-push-out 
  \[
\begin{tikzpicture}
  \matrix (m) [matrix of math nodes,row sep=3em,column sep=4em,minimum width=2em]
  {
     \displaystyle\bigsqcup_{[H]\in I} G\times_{N_{G}[H]}E_{\calF_{n-1}\cap N_{G}[H]}N_{G}[H] & E_{\calF_{n-1}}G \\
      \displaystyle\bigsqcup_{[H]\in I} G\times_{N_{G}[H]}E_{ \calF_{n}[H]}N_{G}[H] & X \\};
  \path[-stealth]
    (m-1-1) edge node [left] {$\displaystyle\bigsqcup_{[H]\in I}id_{G}\times_{N_G}f_{[H]}$} (m-2-1) (m-1-1.east|-m-1-2) edge  node [above] {$i$} (m-1-2)
    (m-2-1.east|-m-2-2) edge node [below] {} (m-2-2)
    (m-1-2) edge node [right] {} (m-2-2);
\end{tikzpicture}
\]
such that $f_{[H]}$ is cellular $N_{G}[H]$-map for every $[H]\in I$ and either (1) $i$ is an inclusion of $G$-CW-complexes, or (2) such that every map $f_{[H]}$ is an inclusion of $N_{G}[H]$-CW-complexes for every $[H]\in I$ and $i$ is a cellular $G$-map. Then $X$ is a model for $E_{\calF_{n}}G$.
 \end{theorem}

 \begin{remark}\label{conditions:exist:pushout} Conditions (1) and (2) at the end of the statements of \cref{Luck:weiermann} are required so that the $G$-push-out is actually a \emph{homotopy} $G$-push-out. It is worth saying that both conditions can be always achieved using a simple \emph{cylinder replacement trick}. For instance, if we want condition (1) to be true, we can replace $E_{\calF_{n-1}}G$ with the mapping cylinder of $i$, which deformation retracts onto the original model for
$E_{\calF_{n-1}}G$ and therefore is again a model for $E_{\calF_{n-1}}G$, and $i$ can be taken to be the natural inclusion. We can do a similar construction if we want condition (2) to hold.
\end{remark}

The following lemma will be also useful.
\begin{lemma}\cite[Lemma~4.4]{DQR11}\label{lemma:union:families}
Let $G$ be a group and $\calF,\, \calG$ be two families of subgroups of $G$.  Choose arbitrary $G$-$CW$-models for $E_{\calF} G$, $E_{\calG}G$ and $E_{\calF\cap\calG} G$. Then, the $G$-$CW$-complex $X$ given by the cellular homotopy $G$-pushout
\[
\xymatrix{
E_{\calF\cap\calG}G \ar[r] \ar[d] & E_{\calF}G \ar[d]\\
E_{\calG}G \ar[r] & X
}
\]
is a model for $E_{\calF\cup\calG}G$.
\end{lemma}

 \section{Commensurators and normalizers of virtually abelian subgroups}\label{section:normalizers}

 This section is devoted to the study of centralizers, normalizers and commensurators of infinite virtually abelian subgroups of $\mcg(S)$ for  a connected compact surface $S$ possibly with a finite number of punctures and boundary components.  The most technical result is \Cref{Normalizer:commensurator:free:abelian:subgroup} which describes, by means of a short exact sequence, the structure of centralizers, normalizers and commensurators of certain abelian subgroups of $\mcg(S)$ for an orientable surface $S$ with empty boundary.

 \vskip 10 pt

 Let $G$ be a group and $H$ a subgroup. We denote $N_G(H)$ the normalizer of $H$ in $G$, and $W_G(H)=N_G(H)/H$ the corresponding Weyl group. 

The centralizer and normalizer in $\mcg(S)$ of a pseudo-Anosov mapping class $f\in\mcg(S)$ are well understood.

\begin{lemma}\label{normalizer:element:pseudoA} Let $S$ be an orientable closed surface with finitely many punctures and negative Euler characteristic. Let $f\in\mcg(S)$ be a pseudo-Anosov mapping class. 
\begin{enumerate}
    \item\cite[Thm. 1]{Mc83} The centralizer $C_{\mcg(S)}(f)$ of $\langle f\rangle$ in  $\mcg(S)$ is a finite extension of an infinite cyclic group. The
normalizer $N_{\mcg(S)}(f)$ of $\langle f\rangle$  in  $\mcg(S)$  is either equal to $C_{\mcg(S)}(f)$  or it contains $C_{\mcg(S)}(f)$ as a  normal subgroup of index $2$.
\item\cite[Prop. 4.8 and Thm. 4.10]{JPT16} The commensurator $N_{\mcg(S)}[\langle f\rangle]$ is a finite extension of an infinite cyclic group. \end{enumerate}
\end{lemma}

It follows from  \cite[Proposition 4.8]{JPT16} and  \cite[Proposition 5.9]{Nucinkis:Petrosyan} that the commensurator of any virtually cyclic subgroup of $\mcg(S)$ is the normalizer of a (finite index) infinite cyclic subgroup.

\begin{proposition}\label{lemma:com:norm:rank1} Let $S$ be an orientable compact surface with finitely many punctures and negative Euler characteristic. Let $H$ be an infinite virtually cyclic subgroup of $\mcg(S)$. Then, there is a finite index cyclic subgroup $H'$ of $H$ such that $N_{\mcg(S)}[H]=N_{\mcg(S)}(H')$.
\end{proposition}

One of the main goals of this section is to prove \cref {thm:normalizer:commensurator:general:case}, the analogous statement for virtually abelian subgroups of $\mcg(S)$ of rank at least 2. This is done in \Cref{norrmalizer:commensurator:equal} when $S$ is a closed surface with negative Euler characteristic,  in \Cref{commensu:equal:normalizer} when $S$ has non-empty boundary and in \Cref{Prop:normalizers:equal:comm:nonorien} for the case when $S$ is non-orientable.

\subsection{Auxiliary short exact sequences}
In this subsection we prove \Cref{Normalizer:commensurator:free:abelian:subgroup}, a key ingredient for the proofs of our \cref {thm:normalizer:commensurator:general:case} and \Cref{Thm:dimension:bound:main}. Furthermore,  it is also used to show in  \Cref{cor:centra:norma} that the centralizer of a free abelian group of rank at least 2 in $\mcg(S)$ has finite index in the corresponding normalizer,  and in \Cref{cor:commensurator:fg} to prove that the commensurators of such subgroups are finitely generated, for the case when $S$ an orientable and closed surface.

\begin{remark}
In  \cite[Proposition 4.12]{JPT16} Juan-Pineda--Trujillo-Negrete obtained a short exact sequence analogous to that in \Cref{Normalizer:commensurator:free:abelian:subgroup}(a) below,  when $H$ is an infinite cyclic subgroup of $\mcg(S)[3]$ generated by a reducible element. Based in this result, Nucinkis--Petrosyan establish in \cite[Proposition 5.6]{Nucinkis:Petrosyan} a short exact sequence analogous to that in \Cref{Normalizer:commensurator:free:abelian:subgroup}(b), again in the infinite cyclic case. 
There is mild enhacement in our generalizations with respect to previous known results of this type: the groups in the right hand side of our short exact sequences, denoted by $Q_i$ for $i=1,2,3$, have finite index in the products $\mcg(\bigsqcup_{i=1}^{a}\hat{S_i})\times A_i$. This is a fundamental ingredient in the proof of \Cref{Thm:dimension:bound:main}, and it might be of independent interest.
\end{remark}

In the following proposition we are considering
the notation established in \Cref{Cor:action:components}.

\begin{proposition}\label{Normalizer:commensurator:free:abelian:subgroup}
 Consider a closed surface $S$ (possibly with punctures) such that $\chi(S)<0$. Let $H$ be  a (free) abelian  subgroup of $\mcg(S)[3]$ of rank  at least $2$, and let  $ \sigma=\sigma(H)=\{[\alpha_1], \dots, [\alpha_n]\}$ be its canonical reduction system. Assume the $H$ acts 
either trivially or via a pseudo-Anosov mapping class in each connected component $\hat{S_i}$ of $\hat S_\sigma$. Then the following statements hold.

\begin{enumerate}[(a)]

    \item There is a central extension 
$$1\to \Z^n \to  C_{\mcg(S)}(H)\xrightarrow[]{\rho_{\sigma}} Q_1\to 1$$
where  $Q_1$ is a finite index subgroup of $ \mcg(\bigsqcup_{i=1}^{a}\hat{S_i})\times A_1$,  $A_1$  is a finitely generated virtually abelian group, and $\rho_\sigma(H)\subseteq A_1$.
    \item There is a central extension 
$$1\to \Z^n \to  N_{\mcg(S)}(H)\xrightarrow[]{\rho_{\sigma}} Q_2\to 1$$
where  $Q_2$ is a finite index subgroup of $ \mcg(\bigsqcup_{i=1}^{a}\hat{S_i})\times A_2$,  $A_2$  is a finitely generated virtually abelian group, and $\rho_\sigma(H)\subseteq A_2$.

    \item There is a central extension
$$1\to \Z^n \to  N_{\mcg(S)}[H]\xrightarrow[]{\rho_{\sigma}} Q_3\to 1$$
where $Q_3$ is a finite index subgroup of $ \mcg(\bigsqcup_{i=1}^{a}\hat{S_i})\times A_3$,  $A_3$  is a finitely generated virtually abelian group, and $\rho_\sigma(H)\subseteq A_3$.
\end{enumerate}
Moreover, $Q_1$ is a finite index subgroup of  $Q_2$, and $Q_2$ is a finite index subgroup of $Q_3$.
\end{proposition}
 \begin{proof}  We split the proof into several steps for the sake of readability.
\vskip 10 pt

\begin{step}\label{step:comm:stab} $C_{\mcg(S)}(H)\subseteq N_{\mcg(S)}(H)\subseteq N_{\mcg(S)}[H]\subseteq \mcg(S)_{\sigma}$.
\end{step}

The first and the second  inclusions are clear. Let $g\in N_{\mcg(S)}[H]$, this means that $gHg^{-1}$ is commensurable with $H$. By \Cref{PropertiesRedSyst} (i) we get  $\sigma(H)=\sigma(gHg^{-1}\cap H)=\sigma(gHg^{-1})$. On the other hand, by \Cref{PropertiesRedSyst} (ii) we have that  $\sigma(gHg^{-1})=g\sigma(H)$. It follows that $g\sigma(H)=\sigma(H)$.

\vskip 10 pt

\begin{step}\label{ses:normalizer:commensurator}
We have three short exact sequences 
\begin{enumerate}[(a)]
\item $1\to \Z^n  \to  C_{\mcg(S)}(H)\xrightarrow[]{\rho_{\sigma}} Q_1\to 1$
where $Q_1=\rho_{\sigma}(C_{\mcg(S)}(H))$.
\item $1\to \Z^n \to  N_{\mcg(S)}(H)\xrightarrow[]{\rho_{\sigma}} Q_2\to 1$
where $Q_2=\rho_{\sigma}(N_{\mcg(S)}(H))$.
\item $1\to \Z^n \to  N_{\mcg(S)}[H]\xrightarrow[]{\rho_{\sigma}} Q_3\to 1$
where $Q_3=\rho_{\sigma}(N_{\mcg(S)}[H])$.
\end{enumerate}
\end{step}

Consider the cutting homomorphism $\rho_{\sigma}\colon \mcg(S)_{\sigma}\to \mcg(\hat{S}_{\sigma})$   
with  $\ker(\rho_{\sigma})=\Z^n$ generated by the Dehn twists along the
curves $\alpha_1,\cdots, \alpha_n$; see \Cref{eq:cutting:homom}. Since each of these Dehn twists centralice $H$, it follows that $\Z^n\subseteq  C_{\mcg(S)}(H)\subseteq  N_{\mcg(S)}(H)\subseteq  N_{\mcg(S)}[H]$. Now \cref{ses:normalizer:commensurator} follows from \Cref{step:comm:stab} by restricting $\rho_\sigma$ to the centralizer, normalizer and commensurator of $H$, respectively.

\vskip 10 pt




 \begin{step}\label{step:normalizer:product}\ 
Using \Cref{Cor:action:components}, we write $\hat{S}_{\sigma}=\bigsqcup_{i=1}^{a}\hat{S}_i \bigsqcup_{j=a+1}^{l}\hat{S}_j$ where $\varphi_i\circ\rho_{\sigma}(H)$ is trivial for all $1\le i\le a$ and $\varphi_i\circ\rho_{\sigma}(H)$ is infinite cyclic generated by a pseudo-Anosov mapping class for all $a+1\le i\le l$. We now show that
\begin{enumerate}[(a)]
    \item $C_{\mcg(\hat{S}_\sigma)}(\rho_{\sigma}(H))= \mcg(\bigsqcup_{i=1}^{a}\hat{S}_i)\times C_{\mcg(\bigsqcup_{j=a+1}^{l}\hat{S}_j)}(\rho_{\sigma}(H))$.
    \item $N_{\mcg(\hat{S}_{\sigma})}(\rho_{\sigma}(H))=\mcg(\bigsqcup_{i=1}^{a}\hat{S}_i)\times N_{\mcg(\bigsqcup_{j=a+1}^{l}\hat{S}_j)}(\rho_{\sigma}(H))$.
     \item $N_{\mcg(\hat{S}_{\sigma})}[\rho_{\sigma}(H)]=\mcg(\bigsqcup_{i=1}^{a}\hat{S}_i)\times N_{\mcg(\bigsqcup_{j=a+1}^{l}\hat{S}_j)}[\rho_{\sigma}(H)]$.
    \end{enumerate}
Moreover all the second factors are finitely generated virtually abelian group of rank $l-a$.
\end{step}

First we show that the following inclusions hold
\begin{equation}\label{eq:normalizer:inside:product}
C_{\mcg(\hat{S}_\sigma)}(\rho_{\sigma}(H))\subseteq N_{\mcg(\hat{S}_{\sigma})}(\rho_{\sigma}(H))\subseteq N_{\mcg(\hat{S}_{\sigma})}[\rho_{\sigma}(H)]\subseteq \mcg(\bigsqcup_{i=1}^{a}\hat{S}_i)\times \mcg(\bigsqcup_{j=a+1}^{l}\hat{S}_j).\end{equation}

    We only prove the third inclusion as the first two inclusions follow from the definitions. Suppose that this is not the case, then there exists $g\in N_{\mcg(\hat{S}_{\sigma})}[\rho_{\sigma}(H)]$ that sends  diffeomorphically $\hat{S}_i$ (for some $1\leq i \leq a$) onto $\hat{S}_j$ (for some $a+1\leq j \leq l$). Since $g\rho_{\sigma}(H)g^{-1}$ is commensurable with $\rho_{\sigma}(H)$ we have that  there is a finite index subgroup $H^{\prime}$ of $H$ such that  for all $x\in \hat{S}_i$, and for all $h'\in H^{\prime}$, we have $\rho_{\sigma}(h')g(x)=g\rho_{\sigma}(h)(x)=g(x)$ for some $h\in H$. Therefore  $\rho_{\sigma}(H')$ acts as the identity on $\hat{S}_j$, which is a contradiction since any finite index subgroup of $\rho_{\sigma}(H)$ must act as a pseudo-Anosov on $\hat{S}_j$. This finishes the proof of the third inclusion. 

Note that, since  $\varphi_i\circ\rho_{\sigma}(H)$ is trivial for all $1\leq i\leq a$, we have $\rho_{\sigma}(H)\subseteq \{ 1 \}\times  \mcg(\bigsqcup_{j=a+1}^{l}\hat{S}_j)$. Hence $ \mcg(\bigsqcup_{i=1}^{a}\hat{S}_i)\times N_{\mcg(\bigsqcup_{j=a+1}^{l}\hat{S}_j)}(\rho_{\sigma}(H))\subseteq N_{\mcg(\hat{S}_{\sigma})}(\rho_{\sigma}(H))$. On the other hand, let $g\in N_{\mcg(\hat{S}_{\sigma})}(\rho_{\sigma}(H))$, by \Cref{eq:normalizer:inside:product} we can write $g=g_1g_2$ with $g_1\in \mcg(\bigsqcup_{i=1}^{a}\hat{S}_i)$ and $g_2\in \mcg(\bigsqcup_{j=a+1}^{l}\hat{S}_j)$. Since $g_1$ centralizes 
$\rho_{\sigma}(H)$, we conclude that $\rho_{\sigma}(H)=g\rho_{\sigma}(H)g^{-1}=g_2\rho_{\sigma}(H)g_2^{-1}$, and therefore $g_2\in N_{\mcg(\bigsqcup_{j=a+1}^{l}\hat{S}_j)}(\rho_{\sigma}(H))$. 
A similar analysis can be carried out for commensurators and centralizers instead of normalizers. From these claims we get the equalities in items (a), (b), and (c).

 Next we prove that $N_{\mcg(\bigsqcup_{j=a+1}^{l}\hat{S}_j)}[\rho_{\sigma}(H)]$ is a finitely generated virtually abelian group. Let $\mcg(\bigsqcup_{j=a+1}^{l}\hat{S}_j)^{0}\cong \prod_{j=a+1}^{l}\mcg(\hat{S}_j,\Omega_j)$ be the finite index subgroup of $\mcg(\bigsqcup_{j=a+1}^{l}\hat{S}_j)$ that fixes each subsurface $\hat{S}_j$ and fixes pointwise the punctures  $\Omega_j$ for $a+1\leq j\leq l$. 
  Since $N_{\mcg(\bigsqcup_{j=a+1}^{l}\hat{S}_j)^{0}}[\rho_{\sigma}(H)]$ has finite index in $N_{\mcg(\bigsqcup_{j=a+1}^{l}\hat{S}_j)}[\rho_{\sigma}(H)]$, it is enough to show that the former is virtually abelian. Let $\varphi_k\colon\mcg(\bigsqcup_{j=a+1}^{l}\hat{S}_j)^{0} \to \mcg(\hat{S}_k,\Omega_k)$ be the canonical projections for $a+1\leq k\leq l$, and consider the following commutative diagrams

\[ \begin{tikzcd}
N_{\mcg(\bigsqcup_{j=a+1}^{l}\hat{S}_j)^{0}}[\rho_{\sigma}(H)] \arrow[hook]{r}{} \arrow[swap]{d}{} & \mcg(\bigsqcup_{j=a+1}^{l}\hat{S}_j)^{0} \arrow{d}{\varphi_k} \\%
N_{\mcg(\hat{S}_k,\Omega_k)}[\varphi_k(\rho_{\sigma}(H))] \arrow[hook]{r}{}& \mcg(\hat{S}_k,\Omega_k).
\end{tikzcd}
\]
Recall that the group $\varphi_k(\rho_{\sigma}(H))$ is  infinite cyclic generated by a pseudo-Anosov mapping class of $\hat{S}_k$, for each $a+1\leq k \leq l$. Therefore the group $ V_k:=N_{\mcg(\hat{S}_k,\Omega_k)}[\varphi_k(\rho_{\sigma}(H))]$ is virtually cyclic by \Cref{normalizer:element:pseudoA}(2). Since $N_{\mcg(\bigsqcup_{j=a+1}^{l}\hat{S}_j)^{0}}[\rho_{\sigma}(H)]$ is a subgroup of finite index of $ \prod_{j=a+1}^{l}V_j$, it follows that $N_{\mcg(\bigsqcup_{j=a+1}^{l}\hat{S}_j)}[\rho_{\sigma}(H)]$ is virtually abelian of rank at most $l-a$. Given that $$C_{\mcg(\bigsqcup_{j=a+1}^{l}\hat{S}_j)}(\rho_{\sigma}(H))\subseteq N_{\mcg(\bigsqcup_{j=a+1}^{l}\hat{S}_j)}(\rho_{\sigma}(H))\subseteq N_{\mcg(\bigsqcup_{j=a+1}^{l}\hat{S}_j)}[\rho_{\sigma}(H)],$$
we conclude that the three groups above are finitely generated virtually abelian of rank at most $l-a$. We finish the proof of this step  by showing that they contain a finitely generated virtually abelian subgroup of rank exactly $l-a$

The group $\prod_{j=a+1}^{l}C_{\mcg(\hat{S}_j,\Omega_j)}(\varphi_j\circ\rho_\sigma(H))$ is  virtually abelian of rank $l-a$. Indeed, the group $\varphi_j\circ\rho_\sigma(H)$ is infinite cyclic generated by a pseudo-Anosov mapping class for all $a\leq j\leq l$. It follows from \cref{normalizer:element:pseudoA}(1) that $C_{\mcg(S_j,\Omega_j)}(\varphi_j\circ\rho_\sigma(H))$ is an infinite virtually cyclic group. Finally, notice that there is a natural  monomorphism  $\prod_{j=a+1}^{l}C_{\mcg(\hat{S}_j,\Omega_j)}(\varphi_j\circ\rho_\sigma(H)) \hookrightarrow C_{\mcg(\bigsqcup_{j=a+1}^{l}\hat{S}_j)^{0}}(\rho_{\sigma}(H))$, and the latter is a subgroup of $C_{\mcg(\bigsqcup_{j=a+1}^{l}\hat{S}_j)}(\rho_{\sigma}(H))$.

\vskip 10 pt

\begin{step}\label{step:Q:finite:index}
The groups $Q_1$, $Q_2$, and $Q_3$ have finite index in $C_{\mcg(\hat{S}_{\sigma})}(\rho_{\sigma}(H))$, $N_{\mcg(\hat{S}_{\sigma})}(\rho_{\sigma}(H))$, and $N_{\mcg(\hat{S}_{\sigma})}[\rho_{\sigma}(H)]$ respectively.
\end{step}

We only prove the second statement for $Q_2:=\rho_{\sigma}(N_{\mcg(S)}(H))\subseteq N_{\mcg(\hat{S}_{\sigma})}(\rho_{\sigma}(H))$ as the proofs of the other two are analogous.

First we show that 
\begin{equation}\label{claim:step4}
    \mcg(\bigsqcup_{i=1}^{a}\hat{S}_i)^{0}\times C_{\mcg(\bigsqcup_{j=a+1}^{l}\hat{S}_j)^{0}}(\rho_{\sigma}(H))\subseteq Q_2.
\end{equation}
Let $f$ be a mapping class in the left hand side of \cref{claim:step4}. 
Following \Cref{subsection:pre-images} we can 
construct a diffeomorphism $\tilde{f}:S\rightarrow S$ representing an element of $\mcg(S)_{\sigma}^{0}$  with image $f$ under the  the cutting homomorphism $\rho_\sigma$.  Such $\tilde{f}$ can be taken as the composite $\tilde{f}_1\tilde{f}_2\cdots\cdots\tilde{f}_l$
of diffeomorphisms of $S$ with disjoint support.  We will show that $\tilde{f}$ represents an element of $ N_{\mcg(S)}(H)$.
 
 Consider $h\in H$. Since $H\subseteq\mcg(S)_{\sigma}^{0}$, we can apply  \Cref{subsection:pre-images} to the mapping class $\rho_{\sigma}(h)$ and construct a  representative of the mapping class $h$ given by a composition    $\tilde{h}_1\tilde{h}_2\cdots\tilde{h}_l$ of diffeomorphims of $S$ where each $\tilde{h}_j$ has support in $S_j$ (considered as a subsurface of $S$). For $1\leq i\leq a$, we have  that $\rho_{\sigma}(h)$ acts trivially on $\hat{S}_i$, therefore we can take $\tilde{h}_i=id_S$.  Moreover, notice that each $\tilde{f}_j$ commutes with the corresponding $\tilde{h}_j$ for $a+1\leq j\leq l$, since $\tilde{f}$ is the lift of an element 
$f$ in $\mcg(\bigsqcup_{i=1}^{a}\hat{S}_i)^{0}\times C_{\mcg(\bigsqcup_{j=a+1}^{l}\hat{S}_j)^{0}}(\rho_{\sigma}(H))$, which is a subgroup of $C_{\mcg(\hat{S}_\sigma)}(\rho_{\sigma}(H))$ by Step 3 above.
Hence the following equalities hold 
\begin{equation*}
\begin{split}
 \tilde{f}(\tilde{h}_1\cdots\tilde{h}_l)\tilde{f}^{-1}&= \tilde{f_1}\tilde{h_1}\tilde{f_1}^{-1}\cdots \tilde{f_a}\tilde{h_a}\tilde{f_a}^{-1}\tilde{f}_{a+1}\tilde{h}_j\tilde{f}_{a+1}^{-1}\cdots\tilde{f}_l\tilde{h}_j\tilde{f}_l^{-1}\\
 &={id}_{S}\cdots {id}_{S} \tilde{h}_{a+1}\cdots \tilde{h}_l\\
 &=\tilde{h}_1\cdots\tilde{h}_l.
\end{split}
\end{equation*}
Therefore the isotopy class of $\tilde{f}$ centralizes any $h\in H$, then it represents an element in $C_{\mcg(S)}(H)\subseteq N_{\mcg(S)}(H)$ and $f\in\rho_{\sigma}( N_{\mcg(S)}(H))$ as desired in \cref{claim:step4}.

From \Cref{step:normalizer:product} and  \cref{claim:step4} we have the following inclusions
\begin{equation*}
    \mcg(\bigsqcup_{i=1}^{a}\hat{S}_i)^{0}\times C_{\mcg(\bigsqcup_{j=a+1}^{l}\hat{S}_j)^{0}}(\rho_{\sigma}(H)) \subset Q_2\subset \mcg(\bigsqcup_{i=1}^{a}\hat{S}_i)\times N_{\mcg(\bigsqcup_{j=a+1}^{l}\hat{S}_j)}(\rho_{\sigma}(H)).
\end{equation*}  
In order to prove that $Q_2$ has finite index in the group in right hand side above, 
it is enough to show that  $C_{\mcg(\bigsqcup_{j=a+1}^{l}\hat{S}_j)^{0}}(\rho_{\sigma}(H))$ has finite index in $N_{\mcg(\bigsqcup_{j=a+1}^{l}\hat{S}_j)}(\rho_{\sigma}(H))$. From \cref{step:normalizer:product},  both $N_{\mcg(\bigsqcup_{j=a+1}^{l}\hat{S}_j)}(\rho_{\sigma}(H))$ and $C_{\mcg(\bigsqcup_{j=a+1}^{l}\hat{S}_j)}(\rho_{\sigma}(H))$ are virtually abelian of rank  $l-a$, therefore the latter has finite index in the former. The conclusion follows since $C_{\mcg(\bigsqcup_{j=a+1}^{l}\hat{S}_j)^{0}}(\rho_{\sigma}(H))$ has finite index in $C_{\mcg(\bigsqcup_{j=a+1}^{l}\hat{S}_j)}(\rho_{\sigma}(H))$.

\vskip 10 pt

The proof of this proposition follows directly from  \Cref{ses:normalizer:commensurator}, \Cref{step:normalizer:product}, and \Cref{step:Q:finite:index}.
\end{proof}

\subsection{Surfaces with empty boundary} Consider  $S$ to be a connected, orientable and closed surface with a finite number of punctures.  We use the auxiliary short exact sequences from the previous subsection to study properties of the centralizers, normalizers and commensurators of certain virtually
abelian subgroups of $\mcg(S)$.


The following corollary is a direct consequence of  \Cref{Normalizer:commensurator:free:abelian:subgroup}.
   
\begin{corollary}\label{cor:centra:norma} 
 Consider a closed orientable connected surface $S$ possibly with a finite number of punctures such that $\chi(S)<0$. Let $H$ be  a (free) abelian subgroup of $\mcg(S)[3]$. Then $C_{\mcg(S)}(H)$ has finite index in $N_{\mcg(S)}(H)$, and $N_{\mcg(S)}(H)$ has finite index in $N_{\mcg(S)}[H]$.
\end{corollary}
 \begin{corollary}\label{cor:commensurator:fg}
Consider a closed orientable connected surface $S$ possibly with a finite number of punctures such that $\chi(S)<0$.  For every virtually abelian subgroup $H$ of $\mcg(S)$ of rank at least 2,  we have that  
  \begin{enumerate}[(a)]
      \item $N_{\mcg(S)}[H]$ is finitely generated.
      \item There is a finite index subgroup $H^{\prime}$ of $H$ such that $N_{\mcg(S)}(H^{\prime})$ is finitely generated.
  \end{enumerate}

 \end{corollary}
\begin{proof}
    First we prove item $(a)$.  By \cref{maximal:rank:abelian:subgroup} the group $H$ is finitely generated.  Let  $H'$ be a finite index (free) abelian subgroup of $H\cap\mcg(S)[3]$, in particular $H'$ is a finite index subgroup of $H$. It follows that  $N_{\mcg(S)}[H]=N_{\mcg(S)}[H^{'}]$. Now, by  \Cref{Normalizer:commensurator:free:abelian:subgroup} (c), it is enough to prove that $Q_3$ is finitely generated. Since $Q_3$ is a finite index subgroup of $\mcg(\bigsqcup_{i=1}^{a}\hat{S}_i)\times A_3$, and clearly $A_3$ and $\mcg(\bigsqcup_{i=1}^{a}\hat{S}_i)$ are finitely generated, it follows that $Q_3$ is finitely generated.
    
Using \Cref{Normalizer:commensurator:free:abelian:subgroup} (b) and choosing $H'$ as before, the proof of part (b) is similar to the proof of the previous item.
\end{proof}

The following definition is borrowed from \cite[Definition~2.8]{tomasz}.
 \begin{definition}(Condition C) Let $n$ be a natural number.
We say that a group $G$ satisfies condition $\mathrm{C}_n$ if for every free abelian subgroup $H$ of $G$ of rank $n$, and for all $K\subset N_{G}[H]$ finitely generated, there is $H^{\prime}$  commensurable with $H$ such that $\langle H, K\rangle \subset N_{G}(H^{\prime})$. Whenever $G$ satisfies condition  $\mathrm{C}_n$, for all $n$, we say $G$ satisfies condition $\mathrm{C}$. 
\end{definition}

\begin{lemma}\label{lemma:conditionC:fg}
 Let $G$ be a group and let $H$ be a free abelian subgroup. Assume that $G$ satisfies condition C, and that $N_G[H]$ is finitely generated. Then there is a finite index subgroup $H'$ of $H$ such that $N_G[H]=N_G(H')$.
\end{lemma}
\begin{proof}
 Applying Condition C to $G$ and $K=N_G[H]$, we get that there is a subgroup $H'\leq H$ of finite index such that $\langle H, N_{G}[H] \rangle\subseteq N_{G}(H')$. On the other hand, $\langle H, N_G[H] \rangle=N_{G}[H]= N_{G}[H']$, hence $N_{G}[H']\subseteq N_{G}(H')$. The other inclusion $N_{G}(H') \subseteq N_{G}[H']$ is always true. 
\end{proof}

\begin{proposition}\label{condition:c} Let $S$ be a closed orientable connected surface possibly with a finite number of punctures.
The mapping class group $\mcg(S)$ satisfies condition $C$.
\end{proposition}

\begin{proof}
Let $H$ be a free abelian group of rank $n$, and let $K\subset N_{G}[H]$ finitely generated. In   \cite[Theorem 1.1.]{MR2271769} it is proved that every solvable subgroup of $\mcg(S)$ is separable, hence $H$ is separable in $\mcg(S)$. Now by \cite[Corollary 9]{MR4138929} there is a finite index subgroup $H^{\prime}$ of $H$ that is normal in $\langle H, K\rangle$. 
\end{proof}

 \begin{theorem}\label{norrmalizer:commensurator:equal}
 Let $S$ be a closed orientable connected surface possibly with a finite number of punctures. Assume that $S$ has negative Euler characteristic. Let $H$ be a virtually abelian  subgroup of $\mcg(S)$ of rank $k\ge 2$. Then there is a subgroup $L\leq H$ of finite index such that $L\leq \mcg(S)[3]$ and $N_{\mcg(S)}[L]=N_{\mcg(S)}(L)$.
\end{theorem}
\begin{proof}  We take  $H^{\prime}$ a finite index (free) abelian subgroup of $H\cap \mcg(S)[3]$, thus $N_G[H]=N_G[H']$. Now the proof follows directly from \Cref{cor:commensurator:fg}, \Cref{condition:c}, and \Cref{lemma:conditionC:fg} applied to $H'$.
\end{proof}

\begin{proposition}\label{vcd:weyl:and:normalizer}
Let $S$ be a closed orientable connected surface possibly with a finite number of punctures. Assume that $S$ has negative Euler characteristic. For every (free) abelian  subgroup $H$ of $\mcg(S)[3]$ of rank $k\ge 1$ we have 
\begin{enumerate}[(a)]
    \item $N_{\mcg(S)}(H)$ and  $W_{\mcg(S)}(H)$ are virtual duality groups.
    \item $\vcd(N_{\mcg(S)}(H))= \vcd(W_{\mcg(S)}(H))+k$.
\end{enumerate}
\end{proposition}

\begin{proof}
First we prove that $W_{\mcg(S)}(H)$ is a virtual duality group. 
Using \cref{Normalizer:commensurator:free:abelian:subgroup}, when the rank of $H$ is at least 2, and \cite[Proposition 5.6]{Nucinkis:Petrosyan}, when $H$ has rank 1, we have the following short exact sequence 
\begin{equation*}
1\to \Z^n \to  N_{\mcg(S)}(H)\xrightarrow[]{\rho_{\sigma}} Q\to 1
\end{equation*}
where  $Q$ is a finite index subgroup of $\mcg(\bigsqcup_{i=1}^{a}\hat{S_i})\times A$, and $A$  is a finitely generated virtually abelian group. From the previous short exact sequence we get the following
\begin{equation*}
1\to  \Z^n/(H\cap\Z^n) \to  W_{\mcg(S)}(H)\xrightarrow[]{\rho_{\sigma}} Q/\rho_{\sigma}(H)\to 1.\end{equation*}
By \cite[Theorem 3.5]{BE73}, to prove that $W_{\mcg(S)}(H)$ is a virtual duality group it is enough to prove that $Q/\rho_{\sigma}(H)$ is a virtual duality group.
Now by \Cref{Normalizer:commensurator:free:abelian:subgroup} we have that 
$Q/\rho_{\sigma}(H)$ is a finite index subgroup of $\mcg(\bigsqcup_{i=1}^{a}\hat{S_i})\times A/\rho_{\sigma}(L)$.  From \cite[Proposition 10.2]{Br94} it follows that finite index subgroups of virtual duality groups are virtual duality groups. Then to prove that $Q/\rho_{\sigma}(H)$  is a virtual duality group it is enough to prove that $\mcg(\bigsqcup_{i=1}^{a}\hat{S_i})\times A/\rho_{\sigma}(L)$  is a virtual duality group. From   \cite[Theorem 3.5]{BE73} it follows  that the finite product of  virtual duality groups is a virtual duality group. Then,   to show  that   $\mcg(\bigsqcup_{i=1}^{a}\hat{S}_i)\times A/\rho_{\sigma}(L)$ is a virtual duality group, it is enough to prove that the factors are virtual duality groups.  The factor $\mcg(\bigsqcup_{i=1}^{a}\hat{S_i})$ has the following finite index subgroup  $\mcg(\bigsqcup_{i=1}^{a}\hat{S_i})^{0}\cong \prod_{i=1}^{a}\mcg(\hat{S}_j)$, thus it is a virtual duality group. The second factor $A/\rho_{\sigma}(L)$ is a finitely generated virtually abelian group as it is a quotient of a finitely generated virtually abelian group. Therefore  $A/\rho_{\sigma}(L)$ is a virtual duality group. 

Since $W_{\mcg(S)}(H)=N_{\mcg(S)}(H)/H$ and $H$ is a free abelian, it follows from \cite[Theorem 3.5]{BE73} that $N_{\mcg(S)}(H)$ is a virtual duality group and item (b) holds.


\end{proof}

\subsection{Surfaces with non-empty boundary}\label{NON-EMPTY}

Let $S$ be a connected surface, possibly with a finite number of punctures, with non-empty boundary. In this section we consider $\hat{S}$ to be the closed surface with punctures obtained from $S$ by `capping' the boundary components with once-punctured disks.  The inclusion $S\hookrightarrow \hat{S}$ induces the {\it capping homomorphism}  \begin{equation*}\theta_S\colon\mcg(S)\rightarrow \mcg(\hat{S},\Omega)\subset\mcg(\hat{S}),\end{equation*}
where $\mcg(\hat{S},\Omega)$ is the subgroup of $\mcg(\hat{S})$ consisting of elements that fix pointwise the set $\Omega$  of punctures of $\hat{S}$ that come from capping the boundary components of $S$.
The homomorphism $\theta_S$ has kernel $\mathbb{Z}^b$ generated by the Dehn twists along curves parallel to the $b>0$ boundary components of $S$.  Since these Dehn twists are central in $\mcg(\hat{S},\Omega)$, we have a central extension $$1\to \Z^b\to \mcg(S)\xrightarrow[]{\theta_S}\mcg(\hat{S},\Omega)\to 1.$$
This central extension is used  to prove the results of interest for surfaces with non-empty boundary from the results obtained in the previous section for surfaces with empty boundary.

\begin{remark}\label{subsection:pre-images capping}[Pre-images of the capping homomorphism] For $f\in \mcg(\hat S,\Omega)$, we denote by  $\tilde{f}\in \mcg(S)$ a pre-image under $\theta_S$ constructed as follows.  Since $f$ fixes $\Omega$ pointwise, for each $x\in \Omega$ we can take a tubular neighborhood $U_x$ such that $\hat{S}-\sqcup_{x\in \Omega} U_x =S$ and there is $F\in f$ that restricts to the identity in $U_x$ for all $x\in \Omega$. We take $\tilde{f}\in\mcg(S)$ to be the isotopy class of $F|_S$; by construction it satisfies that $\theta_S(\tilde{f})=f$. Note that, provided $h\in\mcg(S)$, the element $\widetilde{\theta_S(h)}$ can be chosen to be equal to $h$.
\end{remark}

\begin{proposition}\label{normalizer:boundary}
Let $S$ be a connected compact orientable surface with $b>0$ boundary components and possibly a finite number of punctures.  Let $H$ be a subgroup of $\mcg(S)$, then 
we have the central extension
$$1\to \Z^b\to N_{\mcg(S)}(H)\to N_{\mcg(\hat S,\Omega)}(\theta_S(H))\to 1.$$
\end{proposition}
\begin{proof}
  We can restrict the capping homomorphism to the subgroup $N_{\mcg(S)}(H)\leq\mcg(S)$ to get the following central extension´
  $$1\to \Z^b\to N_{\mcg(S)}(H)\to \theta_S(N_{\mcg(S)}(H))\to 1.$$

   We show that   $\theta_S(N_{\mcg(S)}(H))=N_{\mcg(\hat S,\Omega)}(\theta_S(H))$; we only prove that $N_{\mcg(\hat S,\Omega)}(\theta_S(H))\subseteq \theta_S(N_{\mcg(S)}(H))$ since the other inclusion is clear.  Let  $f\in N_{\mcg(\hat S,\Omega)}(\theta_S(H))$, then for all $h\in H$  we have that $f\theta_S(h)f^{-1}=\theta_S(h')$, for some $h'\in H$. By taking the corresponding pre-images of $\theta_S$ as in  \cref{subsection:pre-images capping} we have that 
   $$\tilde{f}\widetilde{\theta_S(h)}\tilde{f}^{-1}=\widetilde{\theta_S(h')}, \text{ where $\widetilde{\theta_S(h)}=h$ and $\widetilde{\theta_S(h')}=h'$}.$$
  It follows that  $\tilde{f} \in N_{\mcg(S)}(H)$ and $f=\theta_s(\tilde{f})\in\theta_S(N_{\mcg(S)}(H))$.

  \end{proof}

\begin{proposition}\label{vcd:weyl:and:normalizer:boundary}
Let $S$ be a connected orientable compact (possibly with a finite number of punctures) surface with $b>0$ boundary components. Assume that $S$ has negative Euler characteristic.  For every   free abelian  subgroup $H$ of $\mcg(S)$ of rank $k\ge 1$ such that $\theta_S(H)$  is a (free) abelian  subgroup of $\mcg(\hat S,\Omega)\cap \mcg(\hat{S})[3]$, we have 
\begin{enumerate}[(a)]
    \item $N_{\mcg(S)}(H)$ and  $W_{\mcg(S)}(H)$ are virtual duality groups.
    \item $\vcd(N_{\mcg(S)}(H))= \vcd(W_{\mcg(S)}(H))+k$.
\end{enumerate}
\end{proposition}
\begin{proof}
As in the proof of \cref{vcd:weyl:and:normalizer}, it is enough to prove that $W_{\mcg(S)}(H)$ is a virtual duality group. 
By \cref{normalizer:boundary} we get the following short exact sequence 
$$1\to \Z^b/(\Z^b\cap H)\to W_{\mcg(S)}(H)\to W_{\mcg(\hat S,\Omega)}(\theta_S(H))\to 1.$$  
By hypothesis $\theta_S(H)$  is a free abelian  subgroup of $\mcg(\hat S,\Omega)\cap \mcg(\hat{S})[3]$, then by  \cref{vcd:weyl:and:normalizer} $a)$ we have  that $W_{\hat{G}}(\theta_S(H))$ is a virtual duality group. From \cite[Theorem 3.5]{BE73} it follows that  $W_{\mcg(S)}(H)$ is a virtual duality group. 

\end{proof}

\begin{proposition}\label{auxilar:proposition:boundary}
Let $b\ge 1$ be a natural number. Consider a central extension of groups \[1\to \Z^b\to G\xrightarrow[]{\theta} Q\to 1.\]  Suppose that for every finitely generated free abelian subgroup $A$ of $Q$ there is a finite index subgroup $A'$ of $A$ such that $N_Q[A']=N_Q(A')$. Then for every finitely generated free abelian subgroup $H$ of $G$ there is a subgroup $L$ of $G$ such that $L$ is commensurable with $H$ and $N_{G}[L]=N_G(L)$.  
\end{proposition}
\begin{proof}
Let $H$ be a  free abelian subgroup of $G$. Let $r$ be the rank of $\Z^b\cap H$. Since $\Z^b\cap H\le  \Z^b$, then there is unique free abelian subgroup $M$ of $\Z^b$ maximal of rank $r$ such that $\Z^b\cap H \leq M$. Let $T=MH$. Notice that $T$  is a subgroup of $G$ commensurable with $H$, and $\theta (T)=\theta(H)$ is a finitely generated abelian subgroup of $Q$. As a direct consequence of the hypothesis we can find a finite index subgroup $K$ of $\theta(H)$ such that $N_Q[K]=N_Q(K)$.  Let $L=\theta^{-1}(K)\cap T$. Note that $L$ is a finite index subgroup of $T$, which is also commensurable with $H$. We claim that $N_G[L]=N_G(L)$. We only prove that $N_G[L]\subseteq N_G(L)$ as the other inclusion is clear. Let $g\in N_G[L]$ and  $l\in L$. Since that $\theta(N_G[L])\subseteq N_Q[\theta(L)]=N_Q(K)$ we have that $\theta(g)\in N_Q(K)$.  As a consequence we have  $\theta(g)\theta(l)\theta(g)^{-1}= \theta(\tilde{l})$ for some $\tilde{l}\in L$, and therefore $g l g^{-1}=\tilde{l}s$ with $s\in \Z^{b}$. We show  that $s\in M\subset L$, hence we will have that  $g l g^{-1}\in L$, and  $g\in N_G(L)$.  Since $g\in N_G[L]$ we have that $g L g^{-1}\cap L$ has finite index in $g L g^{-1}$ and the extension is central there is $n\in \mathbb{N}$ such that $g l^{n} g^{-1}=\tilde{l}^{n}s^n\in L$. This implies $s^n\in L\cap \Z^b=M$, and the maximality of $M$ gives us that $s\in M$.
\end{proof}

\begin{proposition}\label{normarlizer:equal:Commen:subgroup}
Let $G$ be a group and $K$ a subgroup of $G$. Let $H$ be a subgroup of $K$, suppose that there is a subgroup $L$ of $K$ such that $N_G[H]=N_G(L)$, then  $N_K[H]=N_K(L)$. 
\end{proposition}
\begin{proof}
Note that $N_K[H]=N_G[H]\cap K=N_G(L)\cap K=N_K(L).$
\end{proof}

\begin{proposition}\label{commensu:equal:normalizer}
Let $S$ be a connected compact (possibly with punctures) orientable surface  with $b>0$ boundary components. Assume that $S$ has negative Euler characteristic. Let $H$ be a virtually abelian subgroup of $\mcg(S)$ of rank $k\geq 2$, then there is a subgroup $L$ of $\mcg(S)$  such that $L$ is commensurable with $H$ ,   $N_{\mcg(S)}[L]= N_{\mcg(S)}(L)$ and $\theta_S(L)$ is a subgroup of $\mcg(\hat{S})[3]$.  
\end{proposition}
\begin{proof}
Consider 
the central extension $$1\to \Z^b\to \mcg(S)\xrightarrow[]{\theta_S}\mcg(\hat{S},\Omega)\to 1.$$  

By \cref{norrmalizer:commensurator:equal} we have that for every virtually abelian  subgroup $H$ of $\mcg(\hat S)$ of rank $k\ge 2$, there is a subgroup $L\leq H$ of finite index such that $L\leq \mcg(\hat S)[3]$ and $N_{\mcg(S)}[L]=N_{\mcg(S)}(L)$.
By \cref{normarlizer:equal:Commen:subgroup}, we conclude that for every abelian subgroup $A$ of $\mcg(\hat{S},\Omega)$ there is a finite index subgroup $A'$ of $A$ such that $N_{\mcg(\hat{S},\Omega)}[A']=N_{\mcg(\hat{S},\Omega)}(A').$ The proposition follows from \cref{auxilar:proposition:boundary}.

\end{proof}



\subsection{Non-orientable surfaces}\label{non-orientable:surfaces}
Let $N$ be a non-orientable connected  closed surface possibly with finitely many punctures. Recall that the {\it mapping class group of $N$} is the group $\mcg(N)$ of isotopy classes of diffeomorphisms of $N$. 
Let $p\colon M\to N$ be the orientable double cover of $N$, i.e. $M$ is connected and orientable, and $p$ is a two-sheeted covering map. Assume $N$ has negative Euler characteristic, then 
 there is an injective homomorphism   $\iota \colon \mcg(N) \to \mcg(M)$ (see \cite{BC72} and \cite{GGM18}). In what follows,  
we identify $\mcg(N)$ with the image of $\iota$.


\begin{proposition}
Let $N$ be a connected non-orientable  closed surface with  finitely many punctures. Assume that $N$ has negative Euler characteristic.  Let $\Gamma=\mcg(N)\cap \mcg(M)[3]$, then $\Gamma$ is a torsion-free finite index subgroup of $\mcg(N)$. Moreover,  let $H$ be a free abelian subgroup of $\Gamma$, then $C_{\mcg(N)}(H)$ has finite index in $N_{\mcg(N)}(H)$, and $N_{\mcg(N)}(H)$ has finite index in $N_{\mcg(N)}[H]$.
\end{proposition}
\begin{proof}
Note that, since $N$ has negative Euler characteristic, so does its double cover $M$. Since  $\mcg(M)[3]$ is a torsion-free finite index subgroup of $\mcg(M)$ if follows that $\Gamma=\mcg(N)\cap \mcg(M)[3]$ is a torsion-free finite index subgroup of $\mcg(N)$.   Let $H$ be a free abelian subgroup of $\Gamma$, in particular $H$ is a free abelian subgroup of $\mcg(M)[3]$, hence by \cref{cor:centra:norma},  $C_{\mcg(M)}(H)$ has finite index in $N_{\mcg(M)}(H)$, and $N_{\mcg(M)}(H)$ has finite index in $N_{\mcg(M)}[H]$. Taking the intersection with $\mcg(N)$ respectively we have the claim.
\end{proof}

\begin{proposition}\label{Prop:normalizers:equal:comm:nonorien}
Let $N$ be a connected non-orientable closed surface with  finitely many punctures. Assume that $N$ has negative Euler characteristic. Let $H$ be a virtually abelian subgroup of $\mcg(N)$ of rank $k\geq1$, then there is a finite index subgroup $L$ of $H$ such that $N_{\mcg(N)}[L]=N_{\mcg(N)}(L)$. 
\end{proposition}
\begin{proof}
Since the double cover $M$ of $N$ is an orientable surface with $\chi(M)<0$, then by \Cref{lemma:com:norm:rank1} and \cref{norrmalizer:commensurator:equal} there is a finite index subgroup $L$ of $H$ such that  $N_{\mcg(M)}[L]=N_{\mcg(M)}(L)$. The claim follows from  \cref{normarlizer:equal:Commen:subgroup} \end{proof}

\section{Virtually abelian dimension of mapping class groups}\label{section:virtually:abelian:dimension}

In this section we prove \cref{Thm:dimension:bound:main}.  We first  use the short exact sequences from \cref{Normalizer:commensurator:free:abelian:subgroup} to obtain, in \cref{lemma:auxiliar:dimensions} and \cref{lemma:auxiliar:dimensions:boundary}, upper bounds for geometric dimensions of normalizers and  Weyl groups of some abelian subgroups of $\mcg(S)$.  The upper bounds in \cref{Thm:dimension:bound:main} are then obtained by an induction argument using a L\"uck and Weiermann push-out construction given in \cref{Luck:weiermann} and the push-out of the union of two families given in \cref{lemma:union:families}. The proof relies on the realization of commensurators of abelian subgroups as normalizers obtained in \Cref{thm:normalizer:commensurator:general:case} (see \cref{norrmalizer:commensurator:equal} and \cref{commensu:equal:normalizer}).

\begin{lemma}\label{lemma:auxiliar:dimensions}
 Let $S$ be a connected, closed, oriented  and possibly with a finite number of punctures. Assume that $S$ has negative Euler characteristic, and let  $G=\mcg(S)$. Consider $L$ a free abelian subgroup of $\mcg(S)[3]$ of rank  $k\geq 2$. Define $\calG$ as the smallest family containing  $\{K\subseteq N_G(L)| K\in  \calF_{k}-\calF_{k-1}, K\sim L\}\cup (\calF_0\cap N_G(L))$, and $\calF=\calF_{k-1}\cap N_G(L)$. Let $W_G(L)=N_G(L)/L$. Then
\begin{enumerate}[i)]
    \item $\gd_{\calF_0}(W_G(L))\leq \vcd(W_G(L))$,
    \item $\gd_{\calG}(N_{G}(L))\leq \vcd (W_G(L))$,
    \item $\gd_{\calF \cap \calG}(N_{G}(L)) \leq \vcd (W_G(L))+2k-1$.
\end{enumerate}
\end{lemma}
\begin{proof}
First we prove item $i)$. By  \cref{Normalizer:commensurator:free:abelian:subgroup} we have the following   short exact sequence 
\begin{equation*}
1\to \Z^n \to  N_{G}(L)\xrightarrow[]{\rho_{\sigma}} Q\to 1,
\end{equation*}
where $Q$ is a finite index subgroup of  $\mcg(\bigsqcup_{i=1}^{a}S_i)\times A$  and $A$ is a finitely generated virtually abelian group. Hence we have the next short exact sequence of virtual duality groups
\begin{equation*}
1\to  \Z^n/(L\cap\Z^n) \to  W_G(L)\xrightarrow[]{\rho_{\sigma}} Q/\rho_{\sigma}(L)\to 1,
\end{equation*}
  see the proof of \cref{vcd:weyl:and:normalizer}. Then we have $\vcd(W_G(L))=\vcd(\Z^n/(L\cap\Z^n))+\vcd(Q/\rho_{\sigma}(L))$. Note that the family  $\calF_0$ of $W_G(L)$ is contained in the pullback family $\calH$ of finite subgroups of $Q/\rho_{\sigma}(L)$, thus from \cite[Proposition 2.4]{Nucinkis:Petrosyan}  it is enough to show that   $Q/\rho_{\sigma}(L)$ has a model for $E_{\calF_0}(Q/\rho_{\sigma}(L))$ of dimension $\vcd( Q/\rho_{\sigma}(L))$. Since that   $Q/\rho_{\sigma}(L)\subseteq \mcg(\bigsqcup_{i=1}^{a}S_i)\times A/\rho_{\sigma}(L)$, then 
\begin{equation*}
\begin{split}
    \gd_{\calF_0}(Q/\rho_{\sigma}(L))&\leq \gd_{\calF_0}(\mcg(\bigsqcup_{i=1}^{a}S_i))+ \gd_{\calF_0}(A/\rho_{\sigma}(L))\\
    &= \vcd( \mcg(\bigsqcup_{i=1}^{a}S_i)) + \vcd (A/\rho_{\sigma}(L)), \text{\cite[Prop. 5.3]{Nucinkis:Petrosyan} and \cite[Thm. 5.13]{LW12}}\\
    &=\vcd ( \mcg(\bigsqcup_{i=1}^{a}S_i) \times A/\rho_{\sigma}(L)), \text{as a consequence of \Cref{vcd:weyl:and:normalizer}}\\
    &= \vcd(Q/\rho_\sigma (L)), \text{as a consequence of \Cref{Normalizer:commensurator:free:abelian:subgroup}}.
\end{split}
\end{equation*}
 
Now we prove item $ii)$, i.e.  $\gd_{\calG}(N_{G}(L))\leq \vcd( W_G(L))$. A model for $E_{\calF_0}W_G(L)$ is a model for $E_{\calG}N_{G}(L)$ via the action given by the projection $p\colon N_{G}(L)\to W_{G}(L)$ since the family $\calG$ is the pullback under $p$, of the family of finite subgroups of $W_{G}(L)$, i.e. $\calG$ is the smallest family containing
 $\{p^{-1}(S): S \text{  is a finite subgroup of } W_G(L)\}$.
 Thus it is enough to prove that $W_G(L)$ admits a model for $E_{\calF_0}W_G(L)$ of dimension $\vcd (W_G(L))$. The latter follows from item $i)$.
 
Now we prove item $iii)$, i.e. $\gd_{\calF \cap \calG}(N_{G}(L)) \leq \vcd (W_G(L))+2k-1$.
Applying  \cite[Proposition 5.1 (i)]{LW12} to the inclusion of families $\calF\cap \calG \subset \calG$ we get
\[\gd_{\calF \cap \calG}(N_{G}(L))\leq \gd_\calG(N_G(L))+d\]
for some $d$ such that for any $K\in \calG$ we have $\gd_{\calF\cap \calG\cap K}(K)\leq d$.
Since we already proved $\gd_{\calG}(N_{G}(L))\leq \vcd(W_G(L))$, our next task is to show that $d$ can be chosen to be equal to $2k-1$.

Recall that any $K\in \calG$ is virtually $\Z^t$ for some $0\leq t \leq k$. We split our proof into two cases. First assume that $K\in \calG$ is virtually $\Z^t$ for some $0\leq t \leq k-1$. Hence $K$ belongs to $\calF$, and also belongs to $\calG$ by hypothesis, therefore $K$ belongs to $\calF \cap \calG$ and we conclude $\gd_{\calF\cap\calG\cap K}(K)=0$. Now assume $K\in \calG$ is virtually $\Z^{k}$. We claim that $\calF\cap \calG\cap 
K= \calF_{k-1}\cap K$. The inclusion $\calF\cap \calG\cap K \subset \calF_{k-1}\cap K$ is clear since  $\calF\subset\calF_{k-1}$. For the other inclusion let $M\in \calF_{k-1}\cap K$. Since $K\leq N_G(L)$ we get $\calF_{k-1}\cap K \subseteq \calF_{k-1}\cap N_G(L) =\calF$, on the other hand $M\leq K\in \calG$, therefore $M\in \calF\cap \calG\cap K$. This establishes the claim. We conclude that \[\gd_{\calF\cap \calG\cap K}(K)=\gd_{\calF_{k-1}\cap  K}(K)\leq k+k-1=2k-1\] where the inequality follows from \cite[Proposition 1.3]{tomasz}. 
\end{proof}
\begin{lemma}\label{lemma:auxiliar:dimensions:boundary}
Let $S$ be a connected, compact (possibly with punctures), oriented surface with $b>0$ boundary components. Assume that $S$ has negative Euler characteristic, and let  $G=\mcg(S)$.  Consider $L$ be a free abelian subgroup of $G$ of rank $k\geq 2$ such that $N_G[L]=N_G(L)$ and  $\theta(L)$ is a subgroup of $\mcg(\hat{S})[3]$ (see \cref{commensu:equal:normalizer}). Define $\calG$ as the smallest family containing  $\{K\subseteq N_G(L)| K\in  \calF_{k}-\calF_{k-1}, K\sim L\}$, and $\calF=\calF_{k-1}\cap N_G(L)$. Then
\begin{enumerate}[i)]
    \item $\gd_{\calG}(N_{G}(L))\leq \vcd(W_G(L))$,
    \item $\gd_{\calF \cap \calG}(N_{G}(L)) \leq \vcd(W_G(L))+2k-1$.
\end{enumerate}
\end{lemma}
\begin{proof}
First we show that $\gd_{\calG}(N_{G}(L))\leq \vcd(W_G(L))$. A model for $E_{\calF_0}W_G(L)$ is a model for $E_{\calG}N_{G}(L)$ via the action given by the projection $p\colon N_{G}(L)\to W_{G}(L)$ since the family $\calG$ is the pullback under $p$, of the family of finite subgroups of $W_{G}(L)$.
 Thus it is enough to prove that $W_G(L)$ admits a model for $E_{\calF_0}W_G(L)$ of dimension $\vcd(W_G(L))$.
Denote $\hat G=\mcg(\hat S,\Omega)$  following the notation of section \ref{NON-EMPTY}. By \cref{normalizer:boundary} we get the following short exact sequence 
$$1\to \Z^b/(\Z^b\cap L)\to W_{G}(L)\to W_{\hat{G}}(\theta_S(L))\to 1.$$  
By \cref{vcd:weyl:and:normalizer} $a)$ and \cite[Theorem 3.5]{BE73} it follows that in fact we have a short exact sequence of virtual duality groups, then by \cite[Theorem 3.5]{BE73} we have $\vcd(W_{G}(L))=\vcd( W_{\hat{G}}(\theta_S(L)))+ \vcd(\Z^b/(\Z^b\cap L))$.
Note that $\calF_0$ of $W_G(L)$ is contained in the pullback family  of finite subgroups of $W_{\hat{G}}(\theta_S(L))$, thus from \cite[Proposition 2.4] {Nucinkis:Petrosyan} we get the following 
\begin{equation*}
\begin{split}
    \gd_{\calF_0}(W_{G}(L))&\leq \gd_{\calF_0}(W_{\hat{G}}(\theta_S(L))+\gd_{\calF_0}(\Z^b/(\Z^b\cap L))\\
    &\leq \vcd(W_{\hat{G}}(\theta_S(L)))+\vcd (\Z^b/(\Z^b\cap L)), \text{by \cref{lemma:auxiliar:dimensions} $i)$}\\
    & =\vcd(W_G(L)).
\end{split}
\end{equation*}

The proof of item $ii)$ is completely analogous to the proof of item $iii)$ of \cref{lemma:auxiliar:dimensions}.  
\end{proof}

We now have all the ingredients to prove \cref{Thm:dimension:bound:main}.

\dimensionbound*

\begin{proof}
We use the notation  $G=\mcg(S)$ and proceed by  induction on $n$. 
The case $n=0$ was proved in \cite[Theorem 1.1]{AMP14} for a compact surface without punctures and \cite{Ha86} for a surface compact with a finite number of punctures. The case $n=1$ was proved in  \cite[Theorem 1.4]{Nucinkis:Petrosyan}.
Suppose that the inequality is true for all $n\leq k-1$. We prove that the inequality is also true for $n=k$.
 Let $\sim$ be the equivalence relation on $\calF_{k}-\calF_{k-1}$ defined by commensurability, and let $I$ be  a complete set of representatives of the conjugacy classes in  $(\calF_{k}-\calF_{k-1})/\sim$. Then by \Cref{Luck:weiermann} the following homotopy  $G$-push-out gives a model $X$ of $E_{\calF_{k}}G$
   \[
\begin{tikzpicture}
  \matrix (m) [matrix of math nodes,row sep=3em,column sep=4em,minimum width=2em]
  {
     \displaystyle\bigsqcup_{H\in I} G\times_{N_{G}[H]}E_{N_{G}[H] \cap\calF_{k-1}}N_{G}[H] & E_{\calF_{k-1}}G \\
      \displaystyle\bigsqcup_{H\in I} G\times_{N_{G}[H]}E_{ \calF_{k}[H]}N_{G}[H] & X \\};
  \path[-stealth]
    (m-1-1) edge node [left] {$\displaystyle\bigsqcup_{H\in I}id_{G}\times_{N_G[H]}f_{[H]}$} (m-2-1) (m-1-1.east|-m-1-2) edge  node [above] {} (m-1-2)
    (m-2-1.east|-m-2-2) edge node [below] {} (m-2-2)
    (m-1-2) edge node [right] {} (m-2-2);
\end{tikzpicture}
\]
It follows that
\begin{equation*} 
\begin{split}
\gd_{N_{G}[H] \cap\calF_{k}}&(G)
\leq \dim(X)\\
& \leq \max \{\gd_{N_{G}[H] \cap\calF_{k-1}}(N_{G}[H])+1, \gd_{\calF_{k-1}}(G), \gd_{\calF_{k}[H]}(N_{G}[H])|H\in I\}, \text{ by \Cref{conditions:exist:pushout}} \\
 & \leq \max \{ \gd_{\calF_{k-1}}(G)+1, \gd_{\calF_{k}[H]}(N_{G}[H])|H\in I\}, \text{ since } \gd_{N_{G}[H] \cap\calF_{k-1}}(N_{G}[H])\leq \gd_{\calF_{k-1}}(G)\\
 & \leq \max \{ \vcd(G)+k, \gd_{\calF_{k}[H]}(N_{G}[H])|H\in I\},  \text{ by induction hypothesis.}
\end{split}
\end{equation*}
Hence to prove that  $\gd_{\calF_{k}}(G) \leq \vcd(G)+k$ it is enough to prove that for all $H\in I$    $$\gd_{\calF_{k}[H]}(N_{G}[H])\le  \vcd (G)+k.$$

Let $H\in I$ and take $L$ commensurable with $ H$ as in  \cref{norrmalizer:commensurator:equal} and \cref{commensu:equal:normalizer}, such that $N_{G}[H]=N_{G}[L]=N_{G}(L)$. The family $$\calF_{k}[L]=\{K\subseteq N_G(L)| K\in  \calF_{k}-\calF_{k-1}, K\sim L\}\cup (\calF_{k-1}\cap N_G(L))$$ can be written as the union of two families $\calF_{k}[L]=\calG \cup \calF$, where $\calG$ is the smallest family containig  $\{K\subseteq N_G(L)| K\in  \calF_{k}-\calF_{k-1}, K\sim L\}\cup (\calF_0\cap N_G(L))$, and $\calF=\calF_{k-1}\cap N_G(L)$. By \Cref{lemma:union:families}  the following homotopy  $N_G(L)$-push-out gives a model $Y$ for $E_{\calF_{k}[L]}N_G(L)$

  \[
\begin{tikzpicture}
  \matrix (m) [matrix of math nodes,row sep=3em,column sep=4em,minimum width=2em]
  {
     E_{\calF \cap \calG}N_{G}(L) & E_{\calF}N_{G}(L) \\
      E_{\calG}N_{G}(L) & Y \\};
  \path[-stealth]
    (m-1-1) edge node [left] {} (m-2-1) (m-1-1.east|-m-1-2) edge  node [above] {} (m-1-2)
    (m-2-1.east|-m-2-2) edge node [below] {} (m-2-2)
    (m-1-2) edge node [right] {} (m-2-2);
\end{tikzpicture}
\]
It follows that 
\begin{equation*}
\begin{split}
\gd_{\calF_{k}[H]}(N_{G}[H])&=\gd_{\calF_{k}[L]}(N_{G}(L))\\ &\leq \dim(Y)\\ &\leq\max \{\gd_{\calF}(N_{G}(L)), \gd_{\calF \cap \calG}(N_{G}(L))+1, \gd_{\calG}(N_{G}(L))\} \\
 & \leq \max \{ \gd_{\calF_{k-1}}(G), \gd_{\calF \cap \calG}(N_{G}(L))+1, \gd_{\calG}(N_{G}(L))\}\\
 &\leq \max \{ \gd_{\calF_{k-1}}(G), \vcd (W_G(L))+2k, \vcd (W_G(L))\},\text{ by \Cref{lemma:auxiliar:dimensions} and \cref{lemma:auxiliar:dimensions:boundary}}\\
&\leq \max \{ \gd_{\calF_{k-1}}(G), \vcd (W_G(L))+2k\},\\
&\leq  \max \{ \gd_{\calF_{k-1}}(G), \vcd (N_{G}(L))+k\}, \text{ by \Cref{vcd:weyl:and:normalizer} and \cref{vcd:weyl:and:normalizer:boundary}}\\
&\leq  \max \{ \gd_{\calF_{k-1}}(G), \vcd (G)+k\}, \text{since } \vcd (N_{G}(L))\leq \vcd(G)
\end{split}
\end{equation*}
Finally,  using the induction hypothesis $\gd_{\calF_{k-1}}(G)\leq \vcd(G)+k-1$, we obtain 
\[\gd_{\calF_{k}[H]}(N_{G}[H])  \leq  \vcd(G)+k.\]\end{proof} 


\bibliographystyle{alpha} 
\bibliography{mybib}

\begin{thebibliography}{CCMNP17}

\bibitem[AMP14]{AMP14}
Javier Aramayona and Conchita Mart\'{\i}nez-P\'{e}rez.
\newblock The proper geometric dimension of the mapping class group.
\newblock {\em Algebr. Geom. Topol.}, 14(1):217--227, 2014.

\bibitem[BB19]{BB19}
Arthur Bartels and Mladen Bestvina.
\newblock The {F}arrell-{J}ones conjecture for mapping class groups.
\newblock {\em Invent. Math.}, 215(2):651--712, 2019.

\bibitem[BC72]{BC72}
Joan~S. Birman and D.~R.~J. Chillingworth.
\newblock On the homeotopy group of a non-orientable surface.
\newblock {\em Proc. Cambridge Philos. Soc.}, 71:437--448, 1972.

\bibitem[BE73]{BE73}
Robert Bieri and Beno Eckmann.
\newblock Groups with homological duality generalizing {P}oincar\'{e} duality.
\newblock {\em Invent. Math.}, 20:103--124, 1973.

\bibitem[BLM83]{BirmanLubotzkyMcCarthy83}
Joan~S. Birman, Alex Lubotzky, and John McCarthy.
\newblock Abelian and solvable subgroups of the mapping class groups.
\newblock {\em Duke Math. J.}, 50(4):1107--1120, 1983.

\bibitem[Bri13]{MR3151642}
Martin~R. Bridson.
\newblock On the subgroups of right-angled {A}rtin groups and mapping class
  groups.
\newblock {\em Math. Res. Lett.}, 20(2):203--212, 2013.

\bibitem[Bro94]{Br94}
Kenneth~S. Brown.
\newblock {\em Cohomology of groups}, volume~87 of {\em Graduate Texts in
  Mathematics}.
\newblock Springer-Verlag, New York, 1994.
\newblock Corrected reprint of the 1982 original.

\bibitem[CC05]{MR2150887}
Ruth Charney and John Crisp.
\newblock Automorphism groups of some affine and finite type {A}rtin groups.
\newblock {\em Math. Res. Lett.}, 12(2-3):321--333, 2005.

\bibitem[CCMNP17]{Nucinkis:Moreno:Pasini}
Ged Corob~Cook, Victor Moreno, Brita Nucinkis, and Federico~W. Pasini.
\newblock On the dimension of classifying spaces for families of abelian
  subgroups.
\newblock {\em Homology Homotopy Appl.}, 19(2):83--87, 2017.

\bibitem[CKRW20]{MR4138929}
Pierre-Emmanuel Caprace, Peter~H. Kropholler, Colin~D. Reid, and Phillip
  Wesolek.
\newblock On the residual and profinite closures of commensurated subgroups.
\newblock {\em Math. Proc. Cambridge Philos. Soc.}, 169(2):411--432, 2020.

\bibitem[CW07]{MR2327038}
John Crisp and Bert Wiest.
\newblock Quasi-isometrically embedded subgroups of braid and diffeomorphism
  groups.
\newblock {\em Trans. Amer. Math. Soc.}, 359(11):5485--5503, 2007.

\bibitem[DKP15]{DP15}
D.~Degrijse, R.~K{\"o}hl, and N.~Petrosyan.
\newblock Classifying spaces with virtually cyclic stabilizers for linear
  groups.
\newblock {\em Transform. Groups}, 20(2):381--394, 2015.

\bibitem[DQR11]{DQR11}
James~F. Davis, Frank Quinn, and Holger Reich.
\newblock Algebraic {$K$}-theory over the infinite dihedral group: a controlled
  topology approach.
\newblock {\em J. Topol.}, 4(3):505--528, 2011.

\bibitem[FM12]{FM12}
Benson Farb and Dan Margalit.
\newblock {\em A primer on mapping class groups}, volume~49 of {\em Princeton
  Mathematical Series}.
\newblock Princeton University Press, Princeton, NJ, 2012.

\bibitem[GJP15]{MR3382024}
John Guaschi and Daniel Juan-Pineda.
\newblock A survey of surface braid groups and the lower algebraic {$K$}-theory
  of their group rings.
\newblock In {\em Handbook of group actions. {V}ol. {II}}, volume~32 of {\em
  Adv. Lect. Math. (ALM)}, pages 23--75. Int. Press, Somerville, MA, 2015.

\bibitem[{Gue}22]{Guerch2022}
Yassine {Guerch}.
\newblock {Roots of outer automorphisms of free groups and centralizers of
  abelian subgroups of $\mathrm{Out}(F_N)$}.
\newblock {\em arXiv e-prints}, page arXiv:2212.07674, December 2022.

\bibitem[Har86]{Ha86}
John~L. Harer.
\newblock The virtual cohomological dimension of the mapping class group of an
  orientable surface.
\newblock {\em Invent. Math.}, 84(1):157--176, 1986.

\bibitem[HPa20]{HP20}
Jingyin Huang and Tomasz Prytu\l~a.
\newblock Commensurators of abelian subgroups in {CAT}(0) groups.
\newblock {\em Math. Z.}, 296(1-2):79--98, 2020.

\bibitem[Iva92]{Ivanov:subgroups}
Nikolai~V. Ivanov.
\newblock {\em Subgroups of {T}eichm\"{u}ller modular groups}, volume 115 of
  {\em Translations of Mathematical Monographs}.
\newblock American Mathematical Society, Providence, RI, 1992.
\newblock Translated from the Russian by E. J. F. Primrose and revised by the
  author.

\bibitem[JPTN16]{JPT16}
Daniel Juan-Pineda and Alejandra Trujillo-Negrete.
\newblock On classifying spaces for the family of virtually cyclic subgroups in
  mapping class groups.
\newblock {\em Pure Appl. Math. Q.}, 12(2):261--292, 2016.

\bibitem[Kun19]{kuno2019}
Erika Kuno.
\newblock Abelian subgroups of the mapping class groups for non-orientable
  surfaces.
\newblock {\em Osaka J. Math.}, 56(1):91--100, 2019.

\bibitem[LASSn22]{MR4472883}
Porfirio~L. Le\'{o}n~\'{A}lvarez and Luis~Jorge S\'{a}nchez Salda\~{n}a.
\newblock Classifying spaces for the family of virtually abelian subgroups of
  orientable 3-manifold groups.
\newblock {\em Forum Math.}, 34(5):1277--1296, 2022.

\bibitem[LGGM18]{GGM18}
Daciberg Lima~Gon\c{c}alves, John Guaschi, and Miguel Maldonado.
\newblock Embeddings and the (virtual) cohomological dimension of the braid and
  mapping class groups of surfaces.
\newblock {\em Confluentes Math.}, 10(1):41--61, 2018.

\bibitem[LM07]{MR2271769}
Christopher~J. Leininger and D.~B. McReynolds.
\newblock Separable subgroups of mapping class groups.
\newblock {\em Topology Appl.}, 154(1):1--10, 2007.

\bibitem[LO07]{LO07}
Jean-Fran\c{c}ois Lafont and Ivonne~J. Ortiz.
\newblock Relative hyperbolicity, classifying spaces, and lower algebraic
  {$K$}-theory.
\newblock {\em Topology}, 46(6):527--553, 2007.

\bibitem[LR05]{LR05}
Wolfgang L{\"u}ck and Holger Reich.
\newblock The {B}aum-{C}onnes and the {F}arrell-{J}ones conjectures in {$K$}-
  and {$L$}-theory.
\newblock In {\em Handbook of {$K$}-theory. {V}ol. 1, 2}, pages 703--842.
  Springer, Berlin, 2005.

\bibitem[LW12]{LW12}
Wolfgang L{\"u}ck and Michael Weiermann.
\newblock On the classifying space of the family of virtually cyclic subgroups.
\newblock {\em Pure Appl. Math. Q.}, 8(2):497--555, 2012.

\bibitem[McC83]{Mc83}
John~David McCarthy.
\newblock {\em S{ubgroups} {of} {surface} {mapping} {class} {groups}}.
\newblock ProQuest LLC, Ann Arbor, MI, 1983.
\newblock Thesis (Ph.D.)--Columbia University.

\bibitem[NP18]{Nucinkis:Petrosyan}
Brita Nucinkis and Nansen Petrosyan.
\newblock Hierarchically cocompact classifying spaces for mapping class groups
  of surfaces.
\newblock {\em Bull. Lond. Math. Soc.}, 50(4):569--582, 2018.

\bibitem[Pa21]{tomasz}
Tomasz Prytu\l~a.
\newblock Bredon cohomological dimension for virtually abelian stabilisers for
  {$\rm CAT(0)$} groups.
\newblock {\em J. Topol. Anal.}, 13(3):739--751, 2021.

\end{thebibliography}
\end{document}